\documentclass[a4paper,11pt,twoside]{amsart}
\usepackage[english]{babel}
\usepackage[utf8]{inputenc}

\usepackage[a4paper,inner=3cm,outer=3cm,top=4cm,bottom=4cm,pdftex]{geometry}
\usepackage{fancyhdr}
\usepackage{color}
\usepackage{bold-extra}
\usepackage{ mathrsfs }

\usepackage{graphics}
\usepackage{aliascnt}
\usepackage[hidelinks]{hyperref}

\usepackage{amsmath}
\usepackage{amsfonts}
\usepackage{amssymb}
\usepackage{amsthm}
\usepackage{comment}
\usepackage{mathtools}
\usepackage{enumerate}
\usepackage{enumitem} 

\newtheorem{theorem}{Theorem}[section]
\newtheorem{corollary}[theorem]{Corollary}

\newtheorem{lemma}[theorem]{Lemma}
\newtheorem{proposition}[theorem]{Proposition}
\theoremstyle{definition}
\newtheorem{definition}[theorem]{Definition}
\newtheorem{remark}[theorem]{Remark}
\newtheorem{question}[theorem]{Question}
\newtheorem{example}[theorem]{Example}

\numberwithin{equation}{section}

\newcommand{\dist}{\text{dist}}
\newcommand{\supp}{\text{supp}}

\newcommand\eps{\varepsilon}

\newcommand\R{\mathbb{R}}

\newcommand\s{{\operatorname{s}}}
\newcommand\g{{\operatorname{g}}}
\newcommand\BL{{\operatorname{BL}}}
\newcommand\BLg{{\operatorname{BL_\g}}}
\newcommand\BLs{{\operatorname{BL_\s}}}
\newcommand\D{{\operatorname{D}}}
\newcommand\Dg{{\operatorname{D_\g}}}
\newcommand\ABL{{\operatorname{ABL}}}
\newcommand\ABLg{{\operatorname{ABL_\g}}}
\newcommand\ABLs{{\operatorname{ABL_\s}}}
\newcommand\entropy{{\operatorname{h}}}

\begin{document}

\title{Adjoint Brascamp--Lieb inequalities}

\author[J. Bennett]{Jonathan Bennett}
\address{School of Mathematics, University of Birmingham\\
Edgbaston \\
Birmingham B15 2TT \\
UK}
\email{j.bennett@bham.ac.uk}

\author[T. Tao]{Terence Tao}
\address{Department of Mathematics, UCLA\\
405 Hilgard Ave\\
Los Angeles CA 90095\\
USA}
\email{tao@math.ucla.edu}

\begin{abstract}  The Brascamp--Lieb inequalities are a generalization of the H\"older, Loomis--Whitney, Young, and Finner inequalities that have found many applications in harmonic analysis and elsewhere.  In this paper we introduce an ``adjoint'' version of these inequalities, which can be viewed as an $L^p$ version of the entropic Brascamp--Lieb inequalities of Carlen and Cordero--Erausquin.  As applications, we reprove a log-convexity property of the Gowers uniformity norms, and establish some reverse $L^p$ inequalities for various tomographic transforms.  We conclude with some open questions.
\end{abstract}

\maketitle

\tableofcontents

\section{Introduction}
The Brascamp--Lieb inequalities \cite{BL1976} (also known as the \emph{H\"older--Brascamp--Lieb inequalities}) are a fundamental family of inequalities in analysis, generalizing such classical inequalities as H\"older's inequality, the Loomis--Whitney inequality \cite{Loomis-Whitney}, Young's inequality, and Finner's inequality \cite{finner}.  They have had many recent applications in harmonic analysis, most notably in establishing Vinogradov's mean value conjecture via decoupling inequalities \cite{bdg}; see \cite{zhang} for a recent survey of these developments.  The algorithmic aspects of these inequalities are also related to operator scaling \cite{GGOW2018}, and there are numerous generalizations of the inequalities to other contexts, such as nonlinear inequalities (see e.g., \cite{BBBCF2020}) or to more general quivers \cite{quiver}, as well as a reverse form of the Brascamp--Lieb inequalities due to Barthe \cite{barthe}.  

One can study Brascamp--Lieb inequalities in both continuous and discrete settings:

\begin{definition}[Continuous Brascamp--Lieb constants]\label{cbl-def}  For any surjective linear maps $B_i \colon \R^d \to \R^{d_i}$ between Euclidean spaces and exponents $c_i > 0$ for $i=1,\dots,k$, define the \emph{Brascamp--Lieb constant} $0 < \mathrm{BL}(\mathbf{B}, \mathbf{c}) \leq \infty$ to be the best constant for which one has the \emph{Brascamp--Lieb inequality}
\begin{equation}\label{blc}
 \int_{\R^d} \prod_{i=1}^k f_i^{c_i} \circ B_i \leq \BL(\mathbf{B}, \mathbf{c}) \prod_{i=1}^k \left(\int_{\R^{d_i}} f_i\right)^{c_i}
\end{equation}
for non-negative integrable $f_i \colon \R^{d_i} \to \R$ (where $\mathbf{B} := (B_1,\dots,B_k)$ and $\mathbf{c} := (c_1,\dots,c_k)$).  We refer to the pair $(\mathbf{B}, \mathbf{c})$ as a \emph{(continuous) Brascamp--Lieb datum}.

We define the \emph{gaussian Brascamp--Lieb constant} $\BLg(\mathbf{B}, \mathbf{c})$ similarly, but with the $f_i$ restricted to be (centred) gaussians $f_i(x_i) = C_i e^{-\pi \langle A_i x_i, x_i \rangle}$ for some $C_i>0$ and positive definite $A_i$.
\end{definition}

Of course, the requirement of non-negativity in \eqref{blc} can be dropped by inserting absolute values around the $f_i$ on both sides.

\begin{definition}[Discrete Brascamp--Lieb constants]\label{disc-bl-def}  For any homomorphisms $B_i: G \to G_i$ between discrete abelian groups obeying the non-degeneracy condition that $\bigcap_{i=1}^k \operatorname{ker} B_i$ is finite, and exponents $c_i > 0$ for $i=1,\dots,k$, define the \emph{Brascamp--Lieb constant} $0 < \BL(\mathbf{B}, \mathbf{c}) \leq \infty$ to be the best constant for which one has the \emph{Brascamp--Lieb inequality}
\begin{equation}\label{blcd}
\sum_G \prod_{i=1}^k f_i^{c_i} \circ B_i \leq \BL(\mathbf{B}, \mathbf{c}) \prod_{i=1}^k \Bigl(\sum_{G_i} f_i\Bigr)^{c_i}
\end{equation}
for non-negative absolutely summable $f_i \colon G_i \to \R$, where $\sum_G f$ is shorthand for $\sum_{x \in G} f(x)$. We refer to the pair $(\mathbf{B}, \mathbf{c})$ as a \emph{(discrete) Brascamp--Lieb datum}.

We define the \emph{subgroup Brascamp--Lieb constant} $\BLs(\mathbf{B}, \mathbf{c})$ similarly, but with the $f_i$ restricted to be indicator functions of finite subgroups of $G_i$.
\end{definition}

Note that if $\bigcap_{i=1}^k \operatorname{ker} B_i$ is infinite then $\BL(\mathbf{B}, \mathbf{c})$ will also be infinite, as may be seen by setting all the $f_i$ to be Kronecker delta functions at the origin; the adjoint Brascamp--Lieb constant $\ABLs(\mathbf{B}, \mathbf{c}, \mathbf{\theta}, p)$ defined in Definition \ref{disc-adj} below will similarly be infinite for $0 < p < 1$.  Thus there is little to be lost by imposing the stated non-degeneracy condition.
\begin{example}[H\"older's inequality]\label{Example:Holder}
If $\sum_{i=1}^k c_i=1$ then necessarily $d_i=d$ for all $i$ and the surjections $B_i$ become invertible. The natural changes of variables then convert \eqref{blc} and \eqref{blcd} to classical $k$-linear H\"older inequalities. This yields
$
\BL(\mathbf{B}, \mathbf{c})=\prod_{i=1}^k(\det B_i)^{-c_i}
$ and $\BL(\mathbf{B}, \mathbf{c})=1$ in the continuous and discrete settings respectively.
\end{example} 
\begin{example}[The Loomis--Whitney inequality]\label{Example:LW}
Of particular importance to our applications in Sections \ref{Section:loomis} and \ref{Section:X-ray} is the Loomis--Whitney inequality, which has its origins in \cite{Loomis-Whitney}. This may be formulated in both the continuous and discrete settings. On $\mathbb{R}^d$ it states that
\begin{eqnarray}\label{LWhat}
\begin{aligned}
\int_{\mathbb{R}^d}\prod_{i=1}^d f_i(x_1,\hdots,\widehat{x}_i,\hdots,x_d)^{\frac{1}{d-1}}dx\leq \prod_{i=1}^d\left(\int_{\mathbb{R}^{d-1}}f_i\right)^{\frac{1}{d-1}}
\end{aligned}
\end{eqnarray}
for all measurable functions $f_1,\hdots,f_d:\mathbb{R}^{d-1}\rightarrow\mathbb{R}_+$. Here we use $\widehat{\;\;}$ to denote omission. The underlying linear map $B_ix:=(x_1,\hdots,\widehat{x}_i,\hdots,x_d)$ may be identified with the orthogonal projection $P_{e_i}:\mathbb{R}^d\rightarrow\langle e_i\rangle^\perp$, at which point \eqref{LWhat} becomes the manifestly geometric inequality
\begin{equation*}
\int_{\mathbb{R}^d}\prod_{i=1}^d f_j(P_{e_i}x)^{\frac{1}{d-1}}dx\leq \prod_{i=1}^d\left(\int_{\langle e_i\rangle^\perp}f_i\right)^{\frac{1}{d-1}}
\end{equation*}
for all measurable $f_i:\langle e_i\rangle^\perp\rightarrow\mathbb{R}_+$.
We refer to Section \ref{Section:loomis} for a detailed discussion of the Loomis--Whitney inequality along with its generalization by Finner \cite{finner}.
\end{example}

\begin{example}[Young's convolution inequality on $\mathbb{R}$]\label{Example:Young}
The sharp Young's convolution inequality established in \cite{beckner-rn} and \cite{BL1976} may be interpreted as the Brascamp--Lieb inequality
$$
\int_{\mathbb{R}^{2}}f_1(x)^{c_1}f_2(y)^{c_2}f_3(x-y)^{c_3}dxdy\leq\left(\prod_{i=1}^3\frac{(1-c_i)^{1-c_i}}{c_i^{c_i}}\right)^{\frac{1}{2}}\left(\int_{\mathbb{R}}f_1\right)^{c_1}\left(\int_{\mathbb{R}}f_2\right)^{c_2}\left(\int_{\mathbb{R}}f_3\right)^{c_3}
$$
where $0\leq c_1,c_2,c_3\leq 1$ satisfy $c_1+c_2+c_3=2$.
\end{example}

Of particular importance to the theory of Brascamp--Lieb constants is the celebrated result of Lieb \cite{lieb} (see also \cite{BL1976} for the rank one case), which states that
\begin{equation}\label{lieb-thm}
 \BL(\mathbf{B}, \mathbf{c}) = \BLg(\mathbf{B}, \mathbf{c})
\end{equation}
for any continuous Brascamp--Lieb datum $(\mathbf{B}, \mathbf{c})$.  The discrete analogue
\begin{equation}\label{lieb-thm-disc}
\BL(\mathbf{B}, \mathbf{c}) = \BLs(\mathbf{B}, \mathbf{c})
\end{equation}
of Lieb's theorem for any discrete Brascamp--Lieb datum $(\mathbf{B}, \mathbf{c})$ was obtained by Christ \cite{christ} (see \cite{CDKSY} for further discussion of the torsion-free case, where $\BL(\mathbf{B}, \mathbf{c})$ is equal to $1$ if it is finite).

There are many further useful results on Brascamp--Lieb constants, including characterisations of the data for which they are finite, and for which gaussian extremizers exist; we refer the interested reader to \cite{BCCT2008} and the references there in the first instance. One of the striking features of the theory is the effectiveness of the heat flow monotonicity method (semigroup interpolation) in the context of continuous Brascamp--Lieb inequalities -- see \cite{CLL2004}, \cite{BCCT2008}. As we shall see, the role played by gaussians is rather less fundamental in the setting of the \emph{adjoint} Brascamp--Lieb inequalities that we introduce next.

\subsection{Adjoint Brascamp--Lieb inequalities} 

In this paper we study ``adjoint'' versions of continuous and discrete Brascamp--Lieb inequalities.  To state these adjoint formulations, we need the notion of a pushforward.  If $B \colon G \to H$ is a homomorphism between two discrete abelian groups, we can define the \emph{pushforward} $B_* f \colon H \to \R^+$ of any absolutely summable function $f \colon G \to \R^+$ by the formula
$$ B_* f(h) \coloneqq \sum_{g \in B^{-1}(\{h\})} f(g),$$
or equivalently that
$$ \sum_{h \in H} B_* f(h) F(h) = \sum_{g \in G} f(g) F \circ B(g)$$
for all $F \colon H \to \R^+$.  In a similar vein, if $B \colon \R^d \to \R^{d'}$ is a surjective linear map, we can define the \emph{pushforward} $B_* f \colon \R^{d'} \to \R^+$ of an absolutely integrable function $f \colon \R^d \to \R^+$ by requiring that
\begin{equation}\label{adj-def}
 \int_{\R^{d'}} B_* f(y) F(y)\ dy = \int_{\R^d} f(x) F \circ B(x)\ dx
\end{equation}
for all measurable $F \colon \R^{d'} \to \R^+$.  More explicitly, it follows from a routine change of variables and the Fubini--Tonelli theorem that
\begin{equation}\label{pushform} 
B_* f(y) = \frac{1}{\sqrt{\det(B B^*)}} \int_{B^{-1}(\{y\})} f(x)\ dx
\end{equation}
where the integral on the right-hand side is with respect to surface measure on the $d-d'$-dimensional affine subspace $B^{-1}(\{y\})$ of $\R^d$. We clarify that if $BB^*$ is the identity on $\mathbb{R}^{d'}$ then $B_*f$ may be interpreted as a \emph{marginal} of $f$ with respect to some orthonormal basis of $\mathbb{R}^d$.

We now introduce the following adjoint versions of the continuous and discrete Brascamp--Lieb constants:

\begin{definition}[Continuous adjoint Brascamp--Lieb constants]\label{cts-adj}  For $i=1,\dots,k$, let $B_i: \R^d \to \R^{d_i}$ be surjective linear maps between Euclidean spaces, $\theta_i > 0$, $c_i > 0$, and $0 < p \leq 1$ with $\theta_1+\dots+\theta_k=1$.  Define $0 < p_i \leq 1$ by the formula
\begin{equation}\label{pi-def}
 c_i \left(1-\frac{1}{p}\right) = \theta_i \left(1-\frac{1}{p_i}\right).
\end{equation}
Let $\ABL(\mathbf{B}, \mathbf{c}, \mathbf{\theta}, p)$ denote the best constant such that the \emph{adjoint Brascamp--Lieb inequality}
\begin{equation}\label{abli}
 \| f \|_{L^p(\mathbb{R}^d)} \leq \ABL(\mathbf{B}, \mathbf{c}, \mathbf{\theta}, p) \prod_{i=1}^k \| (B_i)_* f \|_{L^{p_i}(\mathbb{R}^{d_i})}^{\theta_i}
\end{equation}
holds for any non-negative $f \colon \R^d \to \R$.
Define $\ABLg(\mathbf{B}, \mathbf{c}, \mathbf{\theta}, p)$ similarly, but with $f$ restricted to be (centred or uncentred) gaussians.
\end{definition}

\begin{remark}  The condition $\theta_1+\dots+\theta_k=1$ is needed in order for the adjoint Brascamp--Lieb inequality \eqref{abli} to be homogeneous in $f$.  The condition \eqref{pi-def} does not actually constrain the exponents $p, p_i, \theta_i$ appearing in \eqref{abli} (except in the endpoint case $p=1$) since it can be viewed as a formula for the parameters $c_i$; instead, it determines which Brascamp--Lieb inequality the adjoint inequality \eqref{abli} is associated with.  It is irrelevant to the definition of $\ABLg(\mathbf{B}, \mathbf{c}, \mathbf{\theta}, p)$ whether we require $f$ to be a centred gaussian $f(x) = C e^{-\pi \langle A x, x \rangle}$ or an uncentred gaussian $f(x) = C e^{-\pi \langle A (x-x_0), (x-x_0) \rangle}$, since the inequality \eqref{abli} is translation-invariant.
\end{remark}
\begin{remark}\label{Remark:deg}
There are some simple degenerate situations in which \eqref{abli} is an identity, the simplest being when $p=1$. In this case we have $p_i=1$ for all $i$ by \eqref{pi-def}, and since pushforward preserves $L^1$ norms, we have $\ABLg(\mathbf{B}, \mathbf{c}, \mathbf{\theta}, p)=1$ and identical equality in \eqref{abli}. Similarly, if $\mathbf{c}=\mathbf{\theta}$, then $p_i=p$ for all $i$ by \eqref{pi-def}, and since $B_j$ is invertible (see Example \ref{Example:Holder}), we have $(B_i)_*f=(\det B_i)^{-1}f$ for each $i$ by \eqref{pushform}. This also yields identical equality in \eqref{abli} with the appropriate choice of constant $\ABL(\mathbf{B}, \mathbf{c}, \mathbf{\theta}, p)$. In all other cases \eqref{abli} ceases to be an identity.
\end{remark}
\begin{remark}\label{Remark:p<1} In general one cannot expect  to be able to bound $\|f\|_p$ from above in terms of (nontrivial) pushforwards $(B_i)_{*}f$ for any $p>1$.
For example, suppose that $d_i<d$ for all $j$, and select a codimension-1 subspace $V$ of $\mathbb{R}^d$ that is transverse to $\ker B_i$ (i.e. does not contain $\ker B_i$) for every $i$, and let $f(x)=\dist(x,V)^{-1/p}$ on the unit ball, and zero elsewhere. Evidently $f\not\in L^p$, while if $p>1$ then $(B_i)_{*}f$ is bounded and compactly supported. In particular, $(B_i)_{*}f\in L^q$ for any $q$. As may be expected, if $p>1$ a \emph{reverse form} of the adjoint Brascamp--Lieb inequality \eqref{abli} is available -- see Section \ref{adj-sec}.
\end{remark}

\begin{definition}[Discrete adjoint Brascamp--Lieb constants]\label{disc-adj}  For $i=1,\dots,k$, let $B_i: G \to G_i$ be homomorphisms between discrete abelian groups with $\bigcap_{i=1}^k \operatorname{ker} B_i$ finite, $\theta_i > 0$, $c_i > 0$, and $0 < p \leq 1$ with $\theta_1+\dots+\theta_k=1$.  Define $p_i$ by the formula \eqref{pi-def}.
Let $\ABL(\mathbf{B}, \mathbf{c}, \mathbf{\theta}, p)$ denote the best constant such that the \emph{adjoint Brascamp--Lieb inequality}
$$ \| f \|_{\ell^p(G)} \leq \ABL(\mathbf{B}, \mathbf{c}, \mathbf{\theta}, p) \prod_{i=1}^k \| (B_i)_* f \|_{\ell^{p_i}(G_i)}^{\theta_i}$$
holds for any non-negative measurable $f \colon \R^d \to \R$.
Define $\ABLs(\mathbf{B}, \mathbf{c}, \mathbf{\theta}, p)$ similarly, but with $f$ restricted to be the indicator function of a finite subgroup $H$ of $G$.
\end{definition}

We now give our main results relating adjoint Brascamp--Lieb constants to forward Brascamp--Lieb constants.

\begin{theorem}[Main theorem, continuous case]\label{main-cts-thm}  Let $B_i, \theta_i, c_i, p, d, d_i, p_i$ for $i=1,\dots,k$ be as in Definition \ref{cts-adj}.  Then
$$
c(\mathbf{c}, \theta, \mathbf{d}, p) 
\BL(\mathbf{B}, \mathbf{c})^{\frac{1}{p}-1} = \ABLg(\mathbf{B}, \mathbf{c}, \mathbf{\theta}, p)\leq \ABL(\mathbf{B}, \mathbf{c}, \mathbf{\theta}, p) \leq \BL(\mathbf{B}, \mathbf{c})^{\frac{1}{p}-1}$$
where $\mathbf{d} \coloneqq (d,d_1,\dots,d_k)$ and
\begin{equation}\label{const} c(\mathbf{c}, \theta, \mathbf{d}, p) \coloneqq p^{-\frac{d}{2p}} \prod_{i=1}^k p_i^{\frac{\theta_i d_i}{2p_i}}.
\end{equation}
Here we adopt the convention that $0^0 = \infty^0 = 1$.
\end{theorem}

\begin{remark}\label{Remark:iso}
It is possible to pass from the adjoint Brascamp--Lieb inequality
\begin{equation}\label{pto0}
 \| f \|_{L^p(\R^d)} \leq \BL(\mathbf{B}, \mathbf{c})^{\frac{1}{p}-1} \prod_{i=1}^k \| (B_i)_* f \|_{L^{p_i}(\R^{d_i})}^{\theta_i}
\end{equation}
of Theorem \ref{main-cts-thm} back to the Brascamp--Lieb inequality \eqref{blc} by taking a limit as $p\rightarrow 0$. To see this let $f$ be the indicator function of a measurable set $\Omega$ and observe that $\supp((B_i)_*f)\subseteq B_i\Omega$. Letting $p\rightarrow 0$, which forces $p_i\rightarrow 0$ by \eqref{pi-def}, the adjoint inequality \eqref{pto0} becomes
\begin{equation}\label{isop}
|\Omega|\leq \BL(\mathbf{B},\mathbf{c})\prod_{i=1}^k|B_i\Omega|^{c_i}.
\end{equation}
This is simply the Brascamp--Lieb inequality \eqref{blc} applied with $f_i=1_{B_i\Omega}$, which is easily seen to be equivalent to \eqref{blc} for general indicator functions $f_1,\hdots,f_k$. Finally, an application of the tensor power trick (see \cite{BCW2005}, or Section \ref{Section:entropy} for a similar argument) allows one to upgrade \eqref{blc} from indicator functions to general functions. We remark that since pushforwards of indicator functions are typically far from being indicator functions themselves, the adjoint Brascamp--Lieb inequality \eqref{pto0} is not readily interpreted as a functional form of \eqref{isop} for any $p>0$. It is interesting to contrast \eqref{pto0} with the functional forms of \eqref{isop} in \cite[Section 4]{ABBC2021}.
\end{remark}
\begin{remark}\label{Remark:Renyi}
A rather different limiting argument (as $p\rightarrow 1$ now) provides an alternative passage from \eqref{pto0} back to \eqref{blc} via the entropic Brascamp--Lieb inequalities (or generalized entropy subadditivities) of Carlen and Cordero--Erausquin \cite{CCE2009}; see Section \ref{Section:entropy}. It was pointed out to us by Tom Courtade, and independently by Shohei Nakamura and Hiroshi Tsuji, that the adjoint Brascamp--Lieb inequalities \eqref{abli} actually have an entropic formulation for every $p<1$. Specifically, on taking logarithms and applying \eqref{pi-def}, the adjoint Brascamp--Lieb inequalities \eqref{abli} become
\begin{equation}\label{renyi}
\entropy_p(f)\leq\sum_{i=1}^kc_i\entropy_{p_i}((B_i)_*f)+\frac{p}{1-p}\log\ABL(\mathbf{B},\mathbf{c},\mathbf{\theta}, p))
\end{equation}
where $\entropy_p(f)\coloneqq\frac{p}{1-p}\log\|f\|_p$ is the \emph{R\'enyi entropy of order $p$}. Theorem \ref{main-cts-thm} provides the uniform bound $\frac{p}{1-p}\log\ABL(\mathbf{B},\mathbf{c},\mathbf{\theta}, p))\leq \log\BL(\mathbf{B},\mathbf{c})$, from which the entropic Brascamp--Lieb inequalities of Carlen and Cordero--Erausquin are seen to follow in the $p=1$ limit. We refer to \cite{MM2019} and the references there for some related inequalities involving the R\'enyi entropy. Similar remarks may be made in the discrete setting and are left to the interested reader.
\end{remark}

\begin{remark}\label{Remark:degenerate case}
An analysis of the expression \eqref{const} reveals that $c(\mathbf{c}, \theta, \mathbf{d}, p)=1$ if and only if $p=1$ or $\mathbf{c}=\theta$. As discussed in Remark \ref{Remark:deg}, in these degenerate cases \eqref{abli} is an identity, and so evidently
$\ABLg(\mathbf{B}, \mathbf{c}, \mathbf{\theta}, p)= \ABL(\mathbf{B}, \mathbf{c}, \mathbf{\theta}, p).$ As we shall see in Section \ref{Section:ablg-abl}, there is strict inequality here in all other cases. This is in stark contrast with Lieb's theorem, which states that $\BLg(\mathbf{B},\mathbf{c})=\BL(\mathbf{B},\mathbf{c})$ for all data.
\end{remark}

\begin{theorem}[Main theorem, discrete case]\label{main-disc-thm}  Let $B_i, \theta_i, c_i, p, G, G_i, p_i$ for $i=1,\dots,k$ be as in Definition \ref{disc-adj}. Then
$$
\ABLs(\mathbf{B}, \mathbf{c}, \mathbf{\theta}, p) = \ABL(\mathbf{B}, \mathbf{c}, \mathbf{\theta}, p) = \BL(\mathbf{B}, \mathbf{c})^{\frac{1}{p}-1}.$$
\end{theorem}

Combining Theorem \ref{main-cts-thm} with Lieb's theorem \eqref{lieb-thm} and the finiteness theorem from \cite[Theorem 1.13]{BCCT2008} (or \cite[Theorem 2.1]{BCCT2010}) we conclude

\begin{corollary}[Finiteness criterion, continuous case]  Let $(\mathbf{B}, \mathbf{c})$ be a continuous Brascamp--Lieb datum, let $0 < p < 1$, and let $\theta_1,\dots,\theta_k>0$ be real numbers summing to $1$. Then the following are equivalent:
\begin{itemize}
\item[(i)]  $\BL(\mathbf{B}, \mathbf{c})$ is finite.
\item[(ii)]  $\BLg(\mathbf{B}, \mathbf{c})$ is finite.
\item[(iii)] $\ABL(\mathbf{B}, \mathbf{c}, \mathbf{\theta}, p)$ is finite.
\item[(iv)] $\ABLg(\mathbf{B}, \mathbf{c}, \mathbf{\theta}, p)$ is finite.
\item[(v)]  One has
$$ \dim(V) \leq \sum_{i=1}^k c_i \dim( B_i V ),$$
for every subspace $V$ of $\R^d$, with equality when $V=\R^d$.
\end{itemize}
\end{corollary}

Similarly, by combining Theorem \ref{main-disc-thm} with the discrete Lieb theorem \eqref{lieb-thm-disc} and the finiteness theorem from \cite[Theorem 2.4]{BCCT2010}, we have

\begin{corollary}[Finiteness criterion, discrete case]  Let $(\mathbf{B}, \mathbf{c})$ be a discrete Brascamp--Lieb datum, let $0 < p < 1$, and let $\theta_1,\dots,\theta_k>0$ be real numbers summing to $1$. Then the following are equivalent:
\begin{itemize}
\item[(i)]  $\BL(\mathbf{B}, \mathbf{c})$ is finite.
\item[(ii)]  $\BLs(\mathbf{B}, \mathbf{c})$ is finite.
\item[(iii)] $\ABL(\mathbf{B}, \mathbf{c}, \mathbf{\theta}, p)$ is finite.
\item[(iv)] $\ABLs(\mathbf{B}, \mathbf{c}, \mathbf{\theta}, p)$ is finite.
\item[(v)]  One has
$$ \operatorname{rank}(V) \leq \sum_{i=1}^k c_i \operatorname{rank}( B_i V ),$$
for every subgroup $V$ of $G$.
\end{itemize}
\end{corollary}

The inequalities in Theorem \ref{main-cts-thm} are strict in general; we discuss this phenomenon further in Sections \ref{Section:ablg-abl} and \ref{Section:abl-bl}.  One may also take adjoints of further variants of the Brascamp--Lieb inequalities, such as the nonlinear Brascamp--Lieb inequalities; see Section \ref{Section:nonlinear}.

In Section \ref{Section:loomis} we present the simpler special case of adjoint Loomis--Whitney inequalities, which are already of interest.  As one application of these inequalities we reprove a log-convexity property of the Gowers uniformity norms that was previously observed by Shkredov \cite{shkredov} and Manners \cite{manners}. In a further application we obtain some reverse $L^p$ inequalities for tomographic transforms such as the X-ray transform in Section \ref{Section:X-ray}.  Sections \ref{adj-sec} and \ref{Section:lower-adjoint} are then devoted to establishing Theorems \ref{main-cts-thm} and \ref{main-disc-thm} in full generality.  In Section \ref{Section:entropy} we explore the connections between the adjoint Brascamp--Lieb inequalities \eqref{pto0} and the entropic Brascamp--Lieb inequalities of Carlen and Cordero-Erausquin \cite{CCE2009}. Finally, in Section \ref{Section:questions} we pose some questions.
\begin{remark}
The adjoint Brascamp--Lieb inequalities presented here are one of several ``dual" formulations of the Brascamp--Lieb inequalities in the literature. As we have  seen in Remark \ref{Remark:Renyi} (see also Section \ref{Section:entropy}), its closest relatives are the entropic Brascamp--Lieb inequalities of Carlen and Cordero-Erausquin \cite{CCE2009}, which may be viewed as certain Legendre duals of the Brascamp--Lieb inequalities. A rather different duality involving factorization has recently been developed by Carbery, H\"anninen and Valdimarsson \cite{CHV2023}, which has the virtue of also applying to the so-called Kakeya--Brascamp--Lieb inequalities (of which Guth's endpoint multilinear Kakeya inequality \cite{Guth2010} is an example); see also \cite{Ben2021}. Barthe's reverse Brascamp--Lieb inequality \cite{barthe} is a further manifestation of (convex) duality that appears to have no obvious direct connection to ours. The inverse Brascamp--Lieb inequalities of Barthe and Wolff \cite{BW2022} have some similarities with the reverse adjoint Brascamp--Lieb inequalities presented in Section \ref{adj-sec}, although no substantial connection is found. Somewhat further afield, there is a notion of a (Fourier) dual Brascamp--Lieb datum in \cite{BJ2022} that allows one Brascamp--Lieb inequality to be traded for another.
\end{remark}
\subsection{Acknowledgments}

The first author is supported by EPSRC grant EP/W032880/1.
The second author is supported by NSF grant DMS-1764034 and by a Simons Investigator Award.  We thank Tony Carbery, Eric Carlen,  Tom Courtade, Alex Koldobsky, Freddie Manners, Shohei Nakamura, Ron Peled and Hiroshi Tsuji for useful comments.


\section{Adjoint Loomis--Whitney inequalities and the Gowers uniformity norms}\label{Section:loomis}
We begin by demonstrating the process of taking adjoints of a multilinear inequality in a simple case, establishing the adjoint of the well-known Loomis--Whitney inequality \cite{Loomis-Whitney} discussed in Example \ref{Example:LW}.
\begin{theorem}\label{lw-alw}  Let $(\Omega_i,\mu_i)$ be measure spaces for $i=1,\dots,d$.
\begin{itemize}
\item[(i)]  (Loomis--Whitney inequality)  If $f_i \colon \prod_{1 \leq j \leq d; j \neq i} \Omega_j \to \R^+$ is a measurable function for $i=1,\dots,d$, then
\begin{align*}
& \int_{\prod_{j=1}^d \Omega_j} \prod_{i=1}^d f_i( x_1,\dots,\widehat{x}_{i},\dots,x_d)\ d\mu_1(x_1) \dots d\mu_d(x_d) 
 \leq \prod_{i=1}^d \|f_i\|_{L^{d-1}(\prod_{1 \leq j \leq d: j \neq i} \Omega_i)},
\end{align*}
where $\prod_{1 \leq j \leq d: j \neq i} \Omega_i$ is endowed with the product measure $\prod_{1 \leq j \leq d: j \neq i} \mu_i$.
\item[(ii)]  (Adjoint Loomis--Whitney inequality)  If $f \colon \prod_{1 \leq j \leq d} \Omega_j \to \R^+$ is a measurable function, $0 < p \leq 1$, and $\theta_1,\dots,\theta_d > 0$ are such that
\begin{equation}\label{theta-sum} \sum_{i=1}^d \theta_i = 1,
\end{equation}
then
\begin{equation}\label{ano}
 \|f\|_{L^{p}(\prod_{j=1}^d \Omega_j)} \leq \prod_{i=1}^d \| f_i \|_{L^{p_i}(\prod_{1 \leq j \leq d: j \neq i} \Omega_i)}^{\theta_i},
\end{equation}
where $p_i$ is defined by the equation
\begin{equation}\label{od}
 \frac{1}{d-1} \left(1-\frac{1}{p}\right) = \theta_i \left(1-\frac{1}{p_i}\right)
\end{equation}
and
$$ f_i( x_1,\dots,\widehat{x}_{i},\dots,x_d) \coloneqq \int_{\Omega_i} f(x_1,\dots,x_d)\ d\mu_i(x_i)$$
is the $i^{\mathrm{th}}$ marginal of $f$.  For instance, setting all the $\theta_i$ equal to $\frac{1}{d}$, we have
\begin{equation}\label{lw}
 \|f\|_{L^{p}(\prod_{j=1}^d \Omega_j)} \leq \prod_{i=1}^d \| f_i \|_{L^{\frac{p(d-1)}{d-p}}(\prod_{1 \leq j \leq d: j \neq i} \Omega_i)}^{\frac{1}{d}}
\end{equation}
for all $0 < p \leq 1$. 
\item[(iii)]  (Equality)  If in (ii) one has $p<1$, and the two sides of \eqref{ano} are finite and non-zero, then equality holds in \eqref{ano}  if and only if there exist measurable subsets $E_j$ of $\Omega_j$ of positive finite measure for $1,\dots,d$ such that $f(x_1,\dots,x_d) = c \prod_{j=1}^d 1_{E_j}(x_j)$ for almost all $(x_1,\dots,x_d) \in \prod_{j=1}^d \Omega_j$ and some $c>0$.
\end{itemize}
\end{theorem}

\begin{proof}  One can establish (i) by induction and H\"older's inequality: see for instance \cite[Corollary 2.1]{finner}.  To prove (ii), 
we may assume that the factors on the right-hand side of \eqref{ano} are positive and finite, since the claim is trivial otherwise.  

If we write $g_i \coloneqq f_i^{\frac{p_i}{d-1}}$, we see from \eqref{od} that
$$ f_i^{p_i} = f_i g_i^{-\frac{1-p}{\theta_i p}}$$
whenever $f_i \neq 0$, and thus
$$ \int_{\prod_{1 \leq j \leq d: j \neq i} \Omega_j} f_i g_i^{-\frac{1-p}{\theta_i p}}\ \prod_{1 \leq j \leq d: j \neq i} d\mu_j = \|f_i\|_{L^{p_i}(\prod_{1 \leq j \leq d: j \neq i} \Omega_i)}^{p_i},$$
with the understanding that the integrand vanishes when $f_i=0$.  By the Fubini--Tonelli theorem, we conclude that
\begin{equation}\label{id}
 \int_{\prod_{1 \leq j \leq d} \Omega_j} f(x_1,\dots,x_d) g_i( x_1,\dots,\widehat{x}_{i},\dots,x_d)^{-\frac{1-p}{\theta_i p}}\ \prod_{1 \leq j \leq d} d\mu_j(x_j) = \|f_i\|_{L^{p_i}(\prod_{1 \leq j \leq d: j \neq i} \Omega_i)}^{p_i}
\end{equation}
for all $i=1,\dots,d$ (again with the convention that the integrand vanishes when $f=0$).
On the other hand, from Part (i) (with $f_i$ replaced by $g_i$) we have
\begin{equation}\label{lw-conseq}
\int_{\prod_{1 \leq j \leq d} \Omega_j} \prod_{i=1}^d g_i( x_1,\dots,\widehat{x}_{i},\dots,x_d)\ \prod_{1 \leq j \leq d} d\mu_j(x_j) \leq \prod_{i=1}^d \|f_i\|_{L^{p_i}(\prod_{1 \leq j \leq d: j \neq i} \Omega_i)}^{\frac{p_i}{d-1}}.
\end{equation}
Applying H\"older's inequality (raising \eqref{id} to the power $\theta_i p$ and \eqref{lw-conseq} to the power $1-p$, which is permitted thanks to \eqref{theta-sum}) one obtains
$$
 \int_{\prod_{1 \leq j \leq d} \Omega_j} f^p\ \prod_{1 \leq j \leq d} d\mu_j \leq 
\prod_{i=1}^d \|f_i\|_{L^{p_i}(\prod_{1 \leq j \leq d: j \neq i} \Omega_i)}^{\theta_i p p_i + (1-p) \frac{p_i}{d-1}}.
$$
From \eqref{od} and some algebra we see that
$$ \theta_i p p_i + (1-p) \frac{p_i}{d-1} = p p_i \left(\theta_i - \frac{1}{d-1} \left(1-\frac{1}{p}\right)\right) = p p_i \left(\theta_i - \theta_i \left(1 - \frac{1}{p_i}\right)\right) = \theta_i p,$$
giving the desired inequality.

Now we prove (iii).  The ``if'' part of the implication follows from the Fubini--Tonelli theorem and a routine computation, so we focus on the ``only if'' direction.  We must have equality in \eqref{lw-conseq}.  Applying \cite[Theorem 2.1]{finner}, we see that each of the $g_i$ have a tensor product structure $g_i( x_1,\dots,\widehat{x}_{i},\dots,x_d) = \prod_{j \neq i} h_{ij}(x_j)$ almost everywhere for some measurable functions $h_{ij} \colon \Omega_j \to \R^+$.  On the other hand, since the final application of H\"older's inequality in (ii) must also hold with equality, we see that $f^p$ is a scalar multiple of $\prod_{i=1}^d g_i$, thus $f$ also has a tensor product structure
$$ f(x_1,\dots,x_d) = \prod_{j=1}^d h_j(x_j)$$
almost everywhere for some measurable functions $h_j \colon \Omega_j \to \R^+$, which we can take to be in $L^p$ and not zero almost everywhere.  By the Fubini--Tonelli theorem the inequality \eqref{ano} is then the product of the $d$ inequalities
$$  \|h_j\|_{L^{p}(\Omega_j)} \leq \|h_j \|_{L^1(\Omega_j)}^{\theta_j} \prod_{i \neq j} \| h_j \|_{L^{p_i}(\Omega_j)}^{\theta_i},$$
for $j=1,\dots,d$, each of which follows from H\"older's inequality since
$$ \frac{\theta_j}{1} + \sum_{i \neq j} \frac{\theta_i}{p_i} = 1 + \sum_{i \neq j} \theta_i \left(\frac{1}{p_i} - 1\right) =
1 + \sum_{i \neq j} \frac{1}{d-1} \left(\frac{1}{p} - 1\right) = \frac{1}{d}$$
thanks to \eqref{theta-sum} and \eqref{od}.  Since $p \neq 1$, we may apply the converse of H\"older's inequality and conclude that each $h_j$ is a constant multiple of an indicator function $1_{E_j}$ with $E_j$ of finite positive measure, from which the claim readily follows.
\end{proof}

\begin{remark}
In the case $\Omega_j = \R^{d_j}$ when the $\Omega_j$ are Euclidean spaces with Lebesgue measure, the deduction of Theorem \ref{lw-alw}(ii) from Theorem \ref{lw-alw}(i) is a special case of Theorem \ref{main-cts-thm}.  If instead the $\Omega_j$ are discrete abelian groups with counting measure, the deduction is similarly a special case of Theorem \ref{main-disc-thm}.
\end{remark}

\begin{remark}\label{finner-rem} The Loomis--Whitney inequality was generalized by Finner \cite{finner} as follows.  With the notation of Theorem \ref{lw-alw}, suppose that we have a family $S_1,\dots,S_k$ of non-empty subsets of $\{1,\dots,d\}$ and exponents $q_1,\dots,q_k \geq 1$ such that
$$ \sum_{i=1}^k \frac{1}{q_i} 1_{S_i}(j) = 1$$
for all $j=1,\dots,d$.  Setting $\Omega_{S_i} \coloneqq \prod_{j \in S_i} \Omega_j$ with the product measure $\mu_{S_i} \coloneqq \prod_{j \in S_i} \mu_j$, and taking any non-negative $f_i \colon \Omega_{S_i} \to \R^+$, the Finner inequality asserts that
$$ \int_{\prod_{j=1}^d \Omega_j} \prod_{i=1}^k f_i \circ \pi_i\ \prod_{j=1}^d d\mu_j \leq
\prod_{i=1}^k \| f_i \|_{L^{q_i}(\Omega_{S_i})},$$
where $\pi_i \colon \prod_{j=1}^d \Omega_i \to \Omega_{S_i}$ are the projection maps.  Applying the same adjoint procedure used to prove Theorem \ref{lw-alw}, one can conclude that
If $f \colon \prod_{1 \leq j \leq d} \Omega_j \to \R^+$ is measurable, $0 < p \leq 1$, and $\theta_1,\dots,\theta_k > 0$ sum to $1$, then
\begin{equation}\label{ano-finner}
 \|f\|_{L^{p}(\prod_{j=1}^d \Omega_j)} \leq \prod_{i=1}^k \| \pi_i^* f \|_{L^{p_i}(\Omega_{S_i})}^{\theta_i},
\end{equation}
where the exponent $p_i$ is defined by the formula
$$ \frac{1}{q_i} \left(1-\frac{1}{p}\right) = \theta_i \left(1-\frac{1}{p_i}\right)$$
and $(\pi_i)_*$ is the pushforward operator
$$ (\pi_i)_* f( (x_j)_{j \in S_i} ) \coloneqq \int_{\prod_{1 \leq j \leq d: j \not \in S_i} \Omega_j} f( (x_j)_{j=1}^d )\ \prod_{1 \leq j \leq d: j \not \in S_i} d\mu_j.$$
We leave the verification of \eqref{ano-finner} to the interested reader.  Again, this is a special case of Theorem \ref{main-cts-thm} or Theorem \ref{main-disc-thm} when the $\Omega_j$ are Euclidean spaces or discrete abelian groups respectively.  The situation in Theorem \ref{lw-alw} corresponds to the case when $k=d$, $S_i = \{1,\dots,d\}\backslash \{i\}$, and $q_i = \frac{1}{d-1}$.
\end{remark}

We now give an application of this adjoint Loomis--Whitney inequality to the Gowers uniformity norms \cite{gowers}.  Let $(G,+)$ be a locally compact abelian group with a Haar measure $\mu$.  For any non-negative measurable function $f \colon G \mapsto \R^+$ and any $d \geq 1$, we define the Gowers uniformity norm\footnote{This is in fact a seminorm when $d=1$.} $\|f\|_{U^d(G)}$ by the formula
$$ \|f\|_{U^d(G)} \coloneqq \left(\int_{G^{d+1}} \Delta_{h_1} \dots \Delta_{h_d} f(x)\ d\mu(x) d\mu(h_1) \dots d\mu(h_d)\right)^{1/2^d},$$
where $\Delta_h$ denotes the multiplicative derivative
$$ \Delta_h f(x) \coloneqq f(x+h) f(x).$$
One can also define these norms for complex-valued $f$ (if we replace $f(x+h) f(x)$ by $f(x+h) \overline{f(x)}$ in the definition of $\Delta_h f(x)$), but we will not consider this case here.  From H\"older's inequality we easily obtain the inequality
$$ \|f\|_{U^d(G)} \leq \|f\|_{L^{2^d/(d+1)}(G)};$$
see for instance \cite[(5)]{eisner-tao}. In fact the $U^d$ norm also inherits the log-convexity of the $L^{2^d/(d+1)}$ norms:

\begin{corollary}[Log-convexity of Gowers norms]  For any non-negative function $f \coloneqq G \to \R^+$, the function $\frac{d+1}{2^d} \mapsto \|f\|_{U^d(G)}$ for $d=1,2,\dots$ is log-convex.
\end{corollary}

Thus for instance one has
\begin{equation}\label{123}
 \|f\|_{U^2(G)} \leq \|f\|_{U^1(G)}^{1/2} \|f\|_{U^3(G)}^{1/2};
\end{equation}
as a special case, if we set $f=1_A$ for a finite set $A$ and endow $G$ with counting measure, we see (after some routine algebra) that if $A$ contains at least $\delta |A|^3$ parallelograms 
$$(x, x+h, x+k, x+h+k)$$
for some $\delta>0$ (where $x,h,k$ range over $G$), then it contains at least $\delta^4 |A|^4$ parallelopipeds
$$(x, x+h, x+k, x+l, x+h+k, x+h+l,x+k+l,x+h+k+l)$$
(where $x,h,k,l$ range over $G$). This inequality was previously observed by Shkredov \cite[Proposition 35]{shkredov}, with an alternate proof given by Manners \cite[Proposition 2.1]{manners}.  Our proof here can be viewed as the natural generalization of Manners' proof.

\begin{proof}  It suffices to prove the log-convexity for three consecutive values $\|f\|_{U^{d-1}(G)}$, $\|f\|_{U^{d}(G)}$, $\|f\|_{U^{d+1}(G)}$ of the Gowers norms for any $d \geq 2$, that is to say that
$$ \|f\|_{U^d(G)} \leq \|f\|_{U^{d-1}(G)}^\theta  \|f\|_{U^{d+1}(G)}^{1-\theta}$$
where $\theta$ is such that
$$ \frac{d+1}{2^d} = \theta \frac{d}{2^{d-1}} + (1-\theta) \frac{d+2}{2^{d+1}}.$$
We can of course assume that the right-hand side norms are positive and finite.
By scaling both $f$ and the Haar measure $\mu$, we may normalize
$$ \|f\|_{U^{d-1}(G)} = \|f\|_{U^{d+1}(G)} = 1$$
and our goal is now to show that
$$ \|f\|_{U^d(G)} \leq 1.$$
Let $F: G^d \to \R^+$ denote the function
$$ F(h_1,\dots,h_d) := \int_G \Delta_{h_1} \dots \Delta_{h_d} f(x)\ d\mu(x).$$
Then we have
\begin{equation}\label{F-1G}
 \|F\|_{L^1(G^d)} = \|f\|_{U^d(G)}^{2^d}
\end{equation}
and by a routine Fubini and change of variables
\begin{equation}\label{f2}
 \|F\|_{L^2(G^d)} = \|f\|_{U^{d+1}(G)}^{2^d} = 1.
\end{equation}
Also, the marginals $F_i: G^{d-1} \to \R^+$ all take the form
$$ F_i(h_1,\dots,h_{d-1}) = \left(\int_G \Delta_{h_1} \dots \Delta_{h_{d-1}} f(x)\ d\mu(x)\right)^2$$
so we also have
$$ \|F_i\|_{L^{1/2}(G^{d-1})} = \|f\|_{U^{d-1}(G)}^{2^d} = 1.$$
Applying \eqref{lw} with $p = \frac{d}{2d-1}$ (and $\Omega_j=G$ for all $j$), we conclude that
$$ \|F\|_{L^{\frac{d}{2d-1}}(G^d)} \leq 1$$
and hence by \eqref{f2} and H\"older's inequality we also have
$$ \|F\|_{L^1(G^d)} \leq 1,$$
giving the claim from \eqref{F-1G}.
\end{proof}

When $(G,\mu)$ is a probability space, the constant of $1$ implicit in \eqref{123} is sharp, as can be seen by testing against the constant function $f=1$.  However, in the case of the reals $G=\R$ with Lebesgue measure, we believe (in the spirit of the sharp Young and Hausdorff--Young inequalities \cite{beckner}, \cite{BL1976}) that the constant can be improved, thus we conjecture that there exists $\eps>0$ such that
\begin{equation}\label{fu2}
 \|f\|_{U^2(\R)} \leq (1-\eps) \|f\|_{U^1(\R)}^{1/2} \|f\|_{U^3(\R)}^{1/2}
\end{equation}
for all measurable $f \colon \R \to \R^+$, an similarly for higher uniformity norms.  We recall from \cite[Theorem 1.6]{eisner-tao} that we have the sharp inequalities
\begin{align*}
\|f\|_{U^1(\R)} &= \|f\|_{L^1(\R)} \\
\|f\|_{U^2(\R)} &\leq \frac{2^{1/2}}{3^{3/8}} \|f\|_{L^{4/3}(\R)} \\
\|f\|_{U^3(\R)} &\leq \frac{1}{2^{1/8}} \|f\|_{L^2(\R)} 
\end{align*}
and more generally
$$ \|f\|_{U^d(\R)} \leq \frac{2^{d/2^d}}{(d+1)^{(d+1)/2^{d+1}}} \|f\|_{L^{2^d/(d+1)}(\R)}$$
which is consistent with \eqref{fu2} (and its higher degree generalizations) but does not imply it.

While we were not able to establish the conjectured inequality \eqref{fu2}, we have the following partial result as evidence in its favor:

\begin{proposition}[Equality cannot be attained over the reals]\label{no-eq}  Let $f \colon \R \to \R^+$ be such that $\|f\|_{U^1(\R)}$, $\|f\|_{U^3(\R)}$ are non-zero and finite.  Then we have strict inequality
$$ \|f\|_{U^2(\R)} < \|f\|_{U^1(\R)}^{1/2} \|f\|_{U^3(\R)}^{1/2}$$
in \eqref{123}.
\end{proposition}

\begin{proof}  Without loss of generality one can normalize so that
$$ \|f\|_{U^1(\R)} = \|f\|_{U^3(\R)} = 1.$$
Suppose for contradiction that the claim fails, then by \eqref{123} we also have
$$ \|f\|_{U^2(\R)} = 1.$$
If we now let $F \colon \R^2 \to \R^+$ be the function
$$ F(h,k) \coloneqq \int_{\mathbb{R}} \Delta_h \Delta_k f(x)\ dx$$
and let $G \colon \R \to \R^+$ denote the (common) marginal
\begin{equation}\label{gh}
 G(h) = \int_\R F(h,k)dk = \left(\int_\R \Delta_h f(x)\ dx\right)^2,
\end{equation}
then we have
\begin{align}
 \int_{\R} G(h)^{1/2}\ dh &= \|f\|_{U^1(\R)}^2 = 1 \label{g1}\\
 \int_{\R^2} F(h,k)\ dh dk &= \|f\|_{U^2(\R)}^4 = 1 \label{F1}\\
 \int_{\mathbb{R}^2} F^2(h,k)\ dh dk &= \|f\|_{U^3(\R)}^8 = 1.\label{F2}
\end{align}
From \eqref{g1} and \eqref{lw} we have
$$ \int_{\R^2} F^{2/3}(h,k)\ dh dk \leq \left(\int_\R G(h)^{1/2}\right)^{4/3} = 1.$$
On the other hand, from \eqref{F1}, \eqref{F2} and H\"older's inequality we also have
$$ \int_{\R^2} F^{2/3}(h,k)\ dh dk \geq 1.$$
Thus all inequalities here must in fact hold with equality.  By Theorem \ref{lw-alw}(iii), we see that $F(h,k) = c 1_E(h) 1_{E'}(k)$ almost everywhere for some $c>0$ and some positive measure subsets $E,E'$ of $\R$; since $F$ is symmetric we have $E=E'$, and from \eqref{F1}, \eqref{F2} we see that $c=1$ and $E$ has measure one.  This implies that $G = 1_E$ almost everywhere.  From \eqref{gh} we conclude that
$$ f * \tilde f = 1_E$$
almost everywhere, where $\tilde f(x) \coloneqq f(-x)$ is the reflection of $f$ and $*$ denotes the usual convolution operation.  But if we bound $f$ from below by any simple function $g$, $f * \tilde f$ is lower bounded by $g * \tilde g$ which converges to $\|g\|_{L^2(\R)}$ near the origin by the Steinhaus lemma.  Thus $\|g\|_{L^2(\R)} \leq 1$, so by monotone convergence $\|f\|_{L^2(\R)} \leq 1$ (in particular $f$ is square-integrable, which was not immediate from the hypotheses).  This implies that $f * \tilde f$ is continuous (in fact it is in the Wiener algebra), which contradicts the fact that it is an indicator function.
\end{proof}


\section{Upper bounds for adjoint Brascamp--Lieb constants}\label{adj-sec}
Let the notation and hypotheses be as in Theorem \ref{main-cts-thm}.  In this section we establish the inequality
\begin{equation}\label{lo}
 \ABL(\mathbf{B}, \mathbf{c}, \mathbf{\theta}, p) \leq \BL(\mathbf{B}, \mathbf{c})^{\frac{1}{p}-1}.
\end{equation}
A similar argument (which we leave to the reader) gives the corresponding inequality in Theorem \ref{main-disc-thm}.

The proof follows the same lines as that used to prove Theorem \ref{lw-alw}(ii).  Our task is to show that
\begin{equation}\label{fd}
 \| f \|_{L^p(\R^d)} \leq \BL(\mathbf{B}, \mathbf{c})^{\frac{1}{p}-1} \prod_{i=1}^k \| (B_i)_* f \|_{L^{p_i}(\R^{d_i})}^{\theta_i}
\end{equation}
holds for any non-negative measurable $f \colon \R^d \to \R$.  We may assume that the norms on the right-hand side are non-zero and finite, as the claim is trivial otherwise.

If we write $f_i \coloneqq (B_i)_* f$ and $g_i \coloneqq f_i^{c_i p_i}$, we see from \eqref{pi-def} that
$$ f_i^{p_i} = f_i g_i^{-\frac{1-p}{\theta_i p}}$$
whenever $f_i \neq 0$, and thus
$$ \int_{\R^{d_i}} f_i g_i^{-\frac{1-p}{\theta_i p}} = \|f_i\|_{L^{p_i}(\R^{d_i})}^{p_i},$$
with the understanding that the integrand vanishes when $f_i=0$.  Using \eqref{adj-def}, we conclude that
\begin{equation}\label{id2}
 \int_{\R^d} f(x) g_i(B_i x)^{-\frac{1-p}{\theta_i p}}dx = \|f_i\|_{L^{p_i}(\R^{d_i})}^{p_i}
\end{equation}
for all $i=1,\dots,d$ (again with the convention that the integrand vanishes when $f=0$).  
On the other hand, from \eqref{blc} (with $f_i$ replaced by $g_i$) we have
\begin{equation}\label{lw-conseq2}
\int_{\R^d} \prod_{i=1}^k g_i(B_i x)dx \leq \BL(\mathbf{B}, \mathbf{c}) \prod_{i=1}^k \|f_i\|_{L^{p_i}(\R^{d_i})}^{c_i p_i}.
\end{equation}
Applying the $(k+1)$-linear H\"older inequality (raising \eqref{id2} to the power $\theta_i p$ and \eqref{lw-conseq2} to the power $1-p$, which is permitted as $\theta_1+\cdots+\theta_k=1$) one obtains
\begin{equation}\label{afterholder}
 \int_{\R^d} f^p \leq \BL(\mathbf{B}, \mathbf{c})^{1-p}
\prod_{i=1}^k \|f_i\|_{L^{p_i}(\R^{d_i})}^{\theta_i p p_i + (1-p) c_i p_i}.
\end{equation}
From \eqref{pi-def} and some algebra we see that
$$ \theta_i p p_i + (1-p) c_i p_i = p p_i \left(\theta_i - c_i \left(1-\frac{1}{p}\right)\right) = p p_i \left(\theta_i - \theta_i \left(1 - \frac{1}{p_i}\right)\right) = \theta_i p $$
giving the desired inequality.

As mentioned in Remark \ref{Remark:p<1}, for $p>1$ a reverse form of the adjoint Brascamp--Lieb inequality \eqref{abli} is available.
\begin{theorem}[Reverse adjoint Brascamp--Lieb inequality]\label{Theorem:reverse} Suppose $p\geq 1$ and 
$\theta_1,\hdots,\theta_k$ are nonzero real numbers that sum to $1$ and are such that $\theta_i>0$ for precisely one $i$. If \eqref{pi-def} holds then 
\begin{equation}\label{rabli}
\|f\|_p\geq  \BL(\mathbf{B},\mathbf{c})^{\frac{1}{p}-1}\prod_{i=1}^k\|(B_i)_*f\|_{p_i}^{\theta_i}.
\end{equation}
\end{theorem}
\begin{proof}
As \eqref{rabli} holds with equality when $p=1$ we may assume that $p>1$. We follow closely the proof of \eqref{lo}, although after establishing \eqref{lw-conseq2} we apply the $(k+1)$-linear reverse H\"older inequality (see for example \cite{BW2022}) with exponents $\theta_1 p,\hdots,\theta_k p,1-p$, which we may as all but one of these exponents is negative by hypothesis. Since \eqref{lw-conseq2} is raised to the negative power $1-p$, we obtain \eqref{afterholder} with the inequality reversed. The proof of \eqref{rabli} concludes as in the proof of \eqref{lo}.
\end{proof}


The reverse adjoint Brascamp--Lieb inequality \eqref{rabli}, and in particular the associated sign profile of the $\theta_j$, allows us to \emph{transfer bounds between marginals}. For example, consider Loomis--Whitney data with $p=\infty$ (or rather the limiting case as $p\rightarrow\infty$), $\theta_1=\cdots=\theta_{d-1}=-1$ and $\theta_d=d$. Since $c_i=\frac{1}{d-1}$, we have
$p_i=\frac{d-1}{d}$ for $1\leq i\leq d-1$, and $p_d=\frac{d(d-1)}{d(d-1)-1}$. If $\|f\|_\infty=1$, and $f_i$ denotes the $i$th marginal of $f$, then we have
$$
1\geq \prod_{i=1}^{d-1}\|f_i\|_{\frac{d-1}{d}}^{-1}\|f_d\|_{\frac{d(d-1)}{d(d-1)-1}}^d,
$$
or in other words,
$$
\|f_d\|_{\frac{d(d-1)}{d(d-1)-1}}\leq\prod_{i=1}^{d-1}\|f_i\|_{\frac{d-1}{d}}^{\frac{1}{d}}.
$$
We refer to Section \ref{Section:X-ray} for an application of such inequalities to the X-ray transform.


\section{Calculating adjoint Brascamp--Lieb constants for gaussians and subspaces}\label{Section:lower-adjoint}
We now complete the proofs of Theorems \ref{main-cts-thm} and \ref{main-disc-thm}.  As observed in Remark \ref{Remark:deg}, when $p=1$ we have $p_i=1$ for all $i$, and so all adjoint Brascamp--Lieb constants are equal to $1$. Thus we may assume that $0 < p < 1$.

To prove Theorem \ref{main-cts-thm}, it suffices by \eqref{lo}, Lieb's theorem \eqref{lieb-thm}, and the trivial inequality
$$\ABLg(\mathbf{B}, \mathbf{c}, \mathbf{\theta}, p)\leq \ABL(\mathbf{B}, \mathbf{c}, \mathbf{\theta}, p)$$
to establish the equality
\begin{equation}\label{adjoint-calc}
 c(\mathbf{c}, \theta, \mathbf{d}, p) 
\BLg(\mathbf{B}, \mathbf{c})^{\frac{1}{p}-1} = \ABLg(\mathbf{B}, \mathbf{c}, \mathbf{\theta}, p)
\end{equation}
under the hypotheses of that theorem.  Similarly, to complete the proof of Theorem \ref{main-disc-thm} it suffices to establish the identity
\begin{equation}\label{adjoint-calc-disc}
\BLs(\mathbf{B}, \mathbf{c})^{\frac{1}{p}-1} = \ABLs(\mathbf{B}, \mathbf{c}, \mathbf{\theta}, p)
\end{equation}
under the hypotheses of that theorem.

We begin with \eqref{adjoint-calc-disc}, which is simpler.  Let the hypotheses be as in Theorem \ref{main-disc-thm}.  Let $f=1_H$ be the indicator function of a finite subgroup $H$ of $G$.  The arguments in Section \ref{adj-sec} then establish that
$$
 \| f \|_{\ell^p(G)} \leq \BL(\mathbf{B}, \mathbf{c})^{\frac{1}{p}-1} \prod_{i=1}^k \| (B_i)_* f \|_{\ell^{p_i}(G_i)}^{\theta_i}.$$
Note that in the case that $f=1_H$, the functions $g_i \coloneqq ((B_i)_* f)^{p_i}$ are constant multiples of indicator functions $1_{B_i H}$ of the finite subgroups $B_i H$ of $G_i$.  Inspecting the arguments in Section \ref{adj-sec}, we see that we may replace $\BL(\mathbf{B}, \mathbf{c})$ with $\BLs(\mathbf{B}, \mathbf{c})$, thus
$$
 \| f \|_{\ell^p(G)} \leq \BLs(\mathbf{B}, \mathbf{c})^{\frac{1}{p}-1} \prod_{i=1}^k \| (B_i)_* f \|_{\ell^{p_i}(G_i)}^{\theta_i}.$$
Applying Definition \ref{disc-adj}, we conclude that
$$ \ABLs(\mathbf{B}, \mathbf{c}, \mathbf{\theta}, p) \leq \BLs(\mathbf{B}, \mathbf{c})^{\frac{1}{p}-1}$$
which gives the upper bound in \eqref{adjoint-calc-disc}.  To establish the lower bound, it suffices from Definition \ref{disc-bl-def} to show that
\begin{equation}\label{fing}
 \sum_G \prod_{i=1}^k f_i^{c_i} \circ B_i \leq \ABLs(\mathbf{B}, \mathbf{c}, \mathbf{\theta}, p)^{\frac{p}{1-p}} \prod_{i=1}^k \Bigl(\sum_{G_i} f_i\Bigr)^{c_i}
\end{equation}
whenever $f_i = 1_{H_i}$ are indicator functions of finite subgroups $H_i$ of $G$, that is to say that
$$ \# H \leq \ABLs(\mathbf{B}, \mathbf{c}, \mathbf{\theta}, p)^{\frac{p}{1-p}} \prod_{i=1}^k (\# H_i)^{c_i}$$
where $H$ is the subgroup $H \coloneqq \bigcap_{i=1}^k B_i^{-1} H_i$.  From the hypothesis that $\bigcap_{i=1}^k \operatorname{ker} B_i$ is finite, we see that $H$ is finite.   Applying Definition \ref{disc-adj} with $f=1_H$, we have
$$
 \| 1_H \|_{\ell^p(G)} \leq \ABLs(\mathbf{B}, \mathbf{c}, \mathbf{\theta}, p) \prod_{i=1}^k \| (B_i)_* 1_H \|_{\ell^{p_i}(G_i)}^{\theta_i}$$
which we can expand as
$$ (\# H)^{1/p} \leq \ABLs(\mathbf{B}, \mathbf{c}, \mathbf{\theta}, p) \prod_{i=1}^k\left(\frac{\# H}{\#(B_i H)} \#(B_i H)^{1/p_i}\right)^{\theta_i}$$
which rearranges using \eqref{pi-def} as
$$ \# H \leq \ABLs(\mathbf{B}, \mathbf{c}, \mathbf{\theta}, p)^{\frac{p}{1-p}} \prod_{i=1}^k \#(B_i H)^{c_i}.$$
Since $B_i H \subset H_i$, we obtain \eqref{fing} as desired. This completes the proof of \eqref{adjoint-calc-disc}.

Now we turn to \eqref{adjoint-calc}.  If $f_i(x_i) = e^{-\pi \langle A_i x_i, x_i \rangle}$ are centred gaussian functions for some positive-definite $d_i \times d_i$ matrix $A_i$, then as is well known we have
$$ \int_{\R^{d_i}} f_i = \det(A_i)^{-\frac{1}{2}},$$
and similarly
$$ \int_{\R^{d}}\prod_{i=1}^k f^{c_i}_i \circ B_i = \det\Biggl(\sum_{i=1}^k c_i B_i^* A_i B_i \Biggr)^{-\frac{1}{2}}.$$
Hence by Definition \ref{cbl-def},
$$ \BLg(\mathbf{B}, \mathbf{c}) = \sup_{A_i>0} \frac{\prod_{i=1}^k \det(A_i)^{\frac{c_i}{2}}}{\det\left( \sum_{i=1}^k c_i B_i^* A_i B_i \right)^{\frac{1}{2}} },$$
where the supremum is over tuples $(A_1,\dots,A_k)$ of positive-definite $d_i \times d_i$ matrices $A_i$.  Similarly, if $f(x) = e^{-\pi \langle Ax, x \rangle}$ for some positive-definite $d \times d$ matrix $A$, then a routine computation (using for instance the formula for a Fourier transform of a gaussian) shows that for each $i=1,\dots,k$, the pushforward $(B_i)_* f$ is the gaussian function
\begin{equation}\label{pushgauss} (B_i)_* f(x_i) = \frac{\det(A_i)^{\frac{1}{2}}}{\det(A)^{\frac{1}{2}}} e^{-\pi \langle A_i x_i, x_i \rangle},\end{equation}
where $A_i^{-1} \coloneqq B_i A^{-1} B_i^*$.  One can then calculate
$$ \|f\|_{L^p(\R^d)} = p^{-\frac{d}{2p}} \det(A)^{-\frac{1}{2p}}$$
and
$$ \| (B_i)_* f \|_{L^{p_i}(\R^{d_i})} = p_i^{-\frac{d_i}{2p_i}} \frac{\det(A_i)^{\frac{1}{2}}}{\det(A)^{\frac{1}{2}}} \det(A_i)^{-\frac{1}{2p_i}}.$$
From Definition \ref{cts-adj} and \eqref{const} we then see that
\begin{equation}\label{ag}
 \ABLg(\mathbf{B}, \mathbf{c}, \mathbf{\theta}, p) = c(\mathbf{c}, \theta, \mathbf{d}, p)
\sup_{A>0} \frac{\det(A)^{\frac{1}{2} - \frac{1}{2p}}}{\prod_{i=1}^k \det(A_i)^{\frac{\theta_i}{2} - \frac{\theta_i}{2p_i}}}
\end{equation}
where the supremum is over positive-definite $d \times d$ matrices $A$.  Using \eqref{pi-def} and the definition of $A_i$, and then replacing $A$ by $A^{-1}$, we can rewrite this as
$$ \ABLg(\mathbf{B}, \mathbf{c}, \mathbf{\theta}, p) = c(\mathbf{c}, \theta, \mathbf{d}, p)
\sup_{A>0} \frac{\det(A)^{\frac{1}{2p} - \frac{1}{2}}}{\prod_{i=1}^k \det(B_i A B_i^*)^{\frac{c_i}{2p} - \frac{c_i}{2}}},$$
and so \eqref{adjoint-calc} is equivalent to the identity
\begin{equation}\label{ai}
 \sup_{A_i>0} \frac{\prod_{i=1}^k \det(A_i)^{c_i}}{\det\left( \sum_{i=1}^k c_i B_i^* A_i B_i \right) }
=
\sup_{A>0} \frac{\det(A)}{\prod_{i=1}^k \det(B_i A B_i^*)^{c_i}}.
\end{equation}
This identity is somewhat implicit in at least three places in the literature. Most concretely it may be interpreted as the ``entropic'' characterization of (gaussian) Brascamp--Lieb constants from \cite[Section 3.1]{courtade-liu}, as clarified in Remark \ref{Remark:gaussentropic} of the next section; see also the forthcoming Remarks \ref{Remark:bcct} and \ref{Remark:Gurvits} for further context.  For the convenience of the reader we provide a direct proof of \eqref{ai} here.  We may assume that we have the dimensional analysis condition
\begin{equation}\label{ddi}
 d = \sum_{i=1}^k c_i d_i
\end{equation}
since both sides of \eqref{ai} are infinite otherwise.  

If $A$ is an arbitrary positive-definite $d \times d$ matrix, we set $A_i \coloneqq (B_i A B_i^*)^{-1}$ for $i=1,\dots,k$, and observe from the cyclic property of trace that
$$
\operatorname{tr}\left( \sum_{i=1}^k c_i A^{\frac{1}{2}} B_i^* A_i B_i A^{\frac{1}{2}} \right) = \sum_{i=1}^k c_i \operatorname{tr}\left( \mathrm{id}_{\R^{d_i}} \right) = d$$
thanks to \eqref{ddi}.  By the arithmetic mean-geometric mean inequality applied to the eigenvalues of the positive definite matrix $\sum_{i=1}^k c_i A^{\frac{1}{2}} B_i^* A_i B_i A^{\frac{1}{2}}$, we conclude that
$$ \det\left( \sum_{i=1}^k c_i A^{\frac{1}{2}} B_i^* A_i B_i A^{\frac{1}{2}} \right) \leq 1$$
which we can rearrange (using the definition of $A_i$ and the multiplicativity of the determinant) as
$$ \frac{\prod_{i=1}^k \det(A_i)^{c_i}}{\det\left( \sum_{i=1}^k c_i B_i^* A_i B_i \right) } \geq \frac{\det(A)}{\prod_{i=1}^k \det(B_i A B_i^*)^{c_i}}.$$
Taking suprema over $A$, this gives the lower bound in \eqref{ai}.  To establish the corresponding upper bound, let $A_i$ be an arbitrary positive-definite $d_i \times d_i$ matrix for each $i=1,\dots,k$, and set
$$ A \coloneqq \left( \sum_{i=1}^k c_i B_i^* A_i B_i \right)^{-1}.$$
Then by the cyclic property of trace and \eqref{ddi} as before we have
$$
\sum_{i=1}^k c_i \operatorname{tr}\left( A_i^{\frac{1}{2}} B_i A B_i^* A_i^{\frac{1}{2}} \right) = \operatorname{tr}\left( \mathrm{id}_{\R^{d}} \right) = d$$
and hence by the arithmetic mean-geometric mean inequality 
$$
\sum_{i=1}^k \frac{c_i d_i}{d} \det\left( A_i^{\frac{1}{2}} B_i A B_i^* A_i^{\frac{1}{2}} \right)^{\frac{1}{d_i}} \leq 1.$$
By \eqref{ddi} and the (weighted) arithmetic mean-geometric mean inequality, we thus have
$$ \prod_{i=1}^k \det\left( A_i^{\frac{1}{2}} B_i A B_i^* A_i^{\frac{1}{2}} \right)^{\frac{c_i}{d}} \leq 1.$$
This rearranges (using the definition of $A$ and the multiplicativity of the determinant) to
$$ \frac{\prod_{i=1}^k \det(A_i)^{c_i}}{\det\left( \sum_{i=1}^k c_i B_i^* A_i B_i \right) } \leq \frac{\det(A)}{\prod_{i=1}^k \det(B_i A B_i^*)^{c_i}}.$$
Taking suprema over the $A_i$, we obtain the upper bound in \eqref{ai}.  This completes the proof of \eqref{adjoint-calc}.
\begin{remark}\label{Remark:bcct}  A similar calculation shows that if the extremum defining $\BLg(\mathbf{B}, \mathbf{c})$ is attained, then the extremum defining $\ABLg(\mathbf{B}, \mathbf{c}, \mathbf{\theta}, p)$ is attained, and conversely.  Indeed if $A_1,\hdots,A_k$ extremize $\BLg(\mathbf{B}, \mathbf{c})$, it follows from \cite[Proposition 3.6]{BCCT2008} that $A_i^{-1} = B_i A^{-1} B_i^*$ where $A = \sum_i c_i B_i^* A_i B_i$.  Then evidently
$$ \BLg(\mathbf{B}, \mathbf{c}) := \frac{\prod_{i=1}^k \det(A_i)^{\frac{c_i}{2}}}{\det\left(\sum_{i=1}^k c_i B_i^* A_i B_i \right)} = \frac{\prod_{i=1}^k \det( (B_i A^{-1} B_i^*)^{-1} )^{\frac{c_i}{2}}}{\det(A)^{\frac{1}{2}}},$$
from which it follows that $A$ extremizes \eqref{ag} thanks to \eqref{adjoint-calc}. The converse direction is established similarly, as a routine calculation shows that if $A$ extremizes \eqref{ag} then $A = \sum_i c_i B_i^* A_i B_i$ where $A_i^{-1} = B_i A^{-1} B_i^*$.  
Of course the above observations establish the identity \eqref{ai} in the case that the data $(\mathbf{B},\mathbf{c})$ is gaussian extremizable. It is possible to provide such a variational proof of \eqref{ai} for general data using some additional (regularisation) arguments that force gaussian extremizability; see \cite[Section 8]{BCCT2008} for some suitable constructions.
\end{remark}
\begin{remark}\label{Remark:Gurvits}
The proof that we give of the identity \eqref{ai} closely aligns with that of the duality theorem (Proposition 3.14) for capacities of completely positive operators in \cite{GGOW2018}, and indeed \eqref{ai} can be seen to follow from that theorem at least in the case that the exponents $c_1,\hdots, c_k$ are rational; see \cite[Section 4]{GGOW2018}.
\end{remark}


\section{Connections with entropic Brascamp--Lieb inequalities}\label{Section:entropy}
If $f \colon \R^d \to \R^+$ is a probability density (so that $\int_{\R^d} f = 1$) in the $L \log L$ class $\int_{\R^d} f \log(2+f) < \infty$, we define the continuous entropy $\entropy(f)$ by the formula
$$ \entropy(f) \coloneqq \int_{\R^d} f \log \frac{1}{f}$$
In \cite{CCE2009} (see also \cite{CLL2004}) it was observed that for any Brascamp--Lieb datum $(\mathbf{B}, \mathbf{c})$ with finite Brascamp--Lieb constant, one has the inequality\footnote{In \cite{CCE2009} the opposite sign convention $S(f) = -\entropy(f)$ for entropy was used.}
\begin{equation}\label{hf}
 \entropy(f) \leq \sum_{i=1}^k c_i \entropy((B_i)_* f) + \log \BL(\mathbf{B}, \mathbf{c})
\end{equation}
for probability densities $f \colon \R^d \to \R^+$ in the $L \log L$ class, and that the constant $\log \BL(\mathbf{B}, \mathbf{c})$ is best possible.  If one specializes \eqref{hf} to the situation of Loomis--Whitney or Finner data one obtains Shearer's inequality \cite{shearer}.

As indicated in Remark \ref{Remark:Renyi}, the inequality \eqref{hf} may be viewed as an infinitesimal version of the adjoint Brascamp-Lieb inequality \eqref{pto0} in the limit as $p \to 1$ via its R\'enyi-entropic formulation \eqref{renyi}, along with the well-known fact that $\entropy_p(f)\rightarrow\entropy(f)$ as $p\rightarrow 1$ for (suitably regular) probability densities $f$.  Equivalently one may pass from \eqref{pto0} to \eqref{hf} directly as follows. Let $\theta_1,\dots,\theta_k > 0$ be any parameters summing to one, and suppose for simplicity that $f$ is bounded and compactly supported (and thus automatically in $L \log L$).  If $p = 1 - \varepsilon$ for some small $\varepsilon>0$, then from Taylor expansion one has
$$
 \| f \|_{L^p(\R^d)}^p = 1 + \eps \entropy(f) + O(\eps^2)
$$
(where we allow implied constants in the $O()$ notation to depend on all parameters other than $\eps$), and hence also
\begin{equation}\label{taylor}
 \| f \|_{L^p(\R^d)} = 1 + \eps \entropy(f) + O(\eps^2).
\end{equation}
If we define $p_i$ by \eqref{pi-def} then
$$ p_i = 1 - \frac{c_i}{\theta_i} \eps + O(\eps^2)$$
and thus by another Taylor expansion
$$
 \| (B_i)_* f \|_{L^{p_i}(\R^{d_i})} = 1 + \frac{c_i}{\theta_i} \eps \entropy((B_i)_* f) + O(\eps^2).$$
The right-hand side of \eqref{pto0} can then be computed to be
$$ 1 + \eps \log \BL(\mathbf{B}, \mathbf{c}) + \sum_{i=1}^k c_i \eps \entropy((B_i)_* f) + O(\eps^2),$$
and thus on cancelling we obtain
$$ \entropy(f) \leq \sum_{i=1}^k c_i \entropy((B_i)_* f) + \log \BL(\mathbf{B}, \mathbf{c}) + O(\eps).$$
Sending $\eps \to 0$ we obtain \eqref{hf}.  By replacing $f$ with $\min(f,N) / \int \min(f,N)$ for a large $N$ and then sending $N \to \infty$ using dominated convergence, one can then extend the inequality \eqref{hf} to any $f$ of compact support in the $L \log L$ class, and then a further standard limiting argument can remove the compact support hypothesis.

A similar remark can be made for the discrete adjoint Brascamp--Lieb inequality and the discrete version of \eqref{hf} (with continuous entropy replaced by Shannon entropy); we leave the details to the interested reader.

\begin{remark}\label{Remark:gaussentropic}
In addition to $\D(\mathbf{B}, \mathbf{c}):=\log\BL(\mathbf{B}, \mathbf{c})$ being best possible in \eqref{hf}, it is shown in \cite[Section 3.1]{courtade-liu} that the gaussian entropic Brascamp--Lieb constant $\Dg(\mathbf{B}, \mathbf{c})$, defined as the supremum of $$\entropy(f)-\sum_{i=1}^k c_i \entropy((B_i)_* f)$$ over all gaussian probability density functions $f$, coincides with $\D(\mathbf{B}, \mathbf{c})$. By Lieb's theorem \eqref{lieb-thm} and direct computation using \eqref{pushgauss}, this amounts to the key identity \eqref{ai} in Section \ref{Section:lower-adjoint}.
\end{remark}

The above passage from the adjoint Brascamp--Lieb inequality \eqref{pto0} to the entropy inequality \eqref{hf} is of a similar nature to the passage from Nelson's hypercontractivity inequality to the gaussian log-Sobolev inequality discovered by Gross \cite{Gross1975}. In the spirit of \cite{Gross1975} this implication may also be (effectively) reversed, yielding an alternative ``entropic" proof of the adjoint Brascamp--Lieb inequality \eqref{pto0}. We conclude this section with a discussion of this reverse passage, beginning with a key proposition. For technical reasons we work with functions $f$ that are bounded and compactly supported, noting (as is the case above) that the inequalities \eqref{hf} and \eqref{pto0} may be extended to general functions by a standard limiting argument.
\begin{proposition}\label{pdiff} 
If 
$$
\Lambda(p)=\frac{\|f\|_p}{\BL(\mathbf{B},\mathbf{c})^{\frac{1}{p}-1}\prod_{i=1}^k\|f_i\|_{p_i}^{\theta_i}}
$$
then
$$p^2\frac{d}{dp}\log\Lambda (p)=\log\BL(\mathbf{B},\mathbf{c})-\entropy(f^p/\|f\|_p^p)+\sum_{i=1}^k c_i \entropy(f_i^{p_i}/\|f_i\|_{p_i}^{p_i}),$$
\end{proposition}
\begin{proof}
This is an immediate consequence of the elementary identities 
\begin{equation}\label{chain}
\frac{dp_i}{dp}=\frac{c_i}{\theta_i}\left(\frac{p_i}{p}\right)^2
\end{equation} (see \eqref{pi-def}) and 
\begin{equation}\label{normdiff}
p^2\frac{d}{dp}\log\|f\|_p=-\entropy(f^p/\|f\|_p^p).
\end{equation}
\end{proof}
Proposition \ref{pdiff} quickly captures our previous observation that \eqref{pto0}$\implies$\eqref{hf}. One simply observes that
since $\Lambda(p)\leq 1=\Lambda(1)$,
$$\log\BL(\mathbf{B},\mathbf{c})-\entropy(f)+\sum_{i=1}^k c_i \entropy(f_i)=\Lambda'(1)=\lim_{p\rightarrow 1-}\frac{\Lambda(1)-\Lambda(p)}{1-p}\geq 0$$
whenever $f$ is a probability density.
On the other hand, for functions $f$ for which the ``$p$-entropy inequality"
\begin{equation}\label{subadditivityp}
\entropy(f^p/\|f\|_p^p)\leq \log\BL(\mathbf{B},\mathbf{c})+\sum_{i=1}^k c_i \entropy(f_i^{p_i}/\|f_i\|_{p_i}^{p_i})
\end{equation}
holds for all $0<p\leq 1$, Proposition \ref{pdiff} tells us that $\Lambda$ is nondecreasing, recovering \eqref{pto0}.
\begin{remark}
In addition to the classical case when $p=1$, the inequality \eqref{subadditivityp} is true in the limit as $p\rightarrow 0$. To see this observe that
$$
\frac{f^p}{\|f\|_p^p}\rightarrow \frac{1_{\supp(f)}}{|\supp(f)|}
$$
and 
$$\entropy(f^p/\|f\|_p^p)\rightarrow\log\left(|\supp(f)|\right)$$
as $p\rightarrow 0$.
Hence \eqref{subadditivityp}, in the limit as $p\rightarrow 0$, becomes
\begin{equation}\label{p0}
\log\left(|\supp(f)|\right)\leq \log\BL(\mathbf{B},\mathbf{c})+\sum_{i=1}^k c_i\log\left(|\supp(f_i)|\right).
\end{equation}
Now, if $\supp(f)=\Omega$, then $\supp(f_i)\subseteq B_i\Omega$, and so \eqref{p0} becomes
$$
\log\left(|\Omega|\right)\leq \log\BL(\mathbf{B},\mathbf{c})+\sum_{i=1}^k c_i\log\left(|B_i\Omega|\right),
$$
which exponentiates to 
$$
|\Omega|\leq \BL(\mathbf{B},\mathbf{c})\prod_{i=1}^k|B_i\Omega|^{c_i}.
$$
However, this is the Brascamp--Lieb inequality \eqref{blc} applied with $f_i=1_{B_i\Omega}$, as recalled in Remark \ref{Remark:iso}.
\end{remark}
\begin{proposition}\label{prop:hp}
If $0<p\leq 1$ then \eqref{subadditivityp} holds whenever $f$ is an indicator function of a set.
\end{proposition}
\begin{proof}
If $f$ is the indicator function of a set then $\entropy(f^p/\|f\|_p^p)=\entropy(f/\|f\|_1)$, and so by \eqref{hf},
$$
\entropy(f^p/\|f\|_p^p)\leq\log\BL(\mathbf{B},\mathbf{c})+\sum_{i=1}^k c_i \entropy(f_i/\|f_i\|_1).
$$
By \eqref{pi-def} we have $p_i\leq 1$ for all $i$, and so it is enough to show that $p\mapsto \entropy(f^p/\|f\|_p^p)$ is nonincreasing.
This may be done by establishing the formula $$-p\frac{d}{dp}\entropy(f^p/\|f\|_p^p)=\int \frac{f^p}{\|f\|_p^p}\left(\log\left(\frac{f^p}{\|f\|_p^p}\right)\right)^2-\left(\int\frac{f^p}{\|f\|_p^p}\log\left(\frac{f^p}{\|f\|_p^p}\right)\right)^2,$$
which may be interpreted as a certain variance with respect to the probability measure $f^p/\|f\|_p^p$.
Alternatively we may make further use of \eqref{taylor} to write 
$$\frac{\|f\|_{1+\varepsilon}}{\|f\|_1}= 1-\varepsilon \entropy(f/\|f\|_1)+O(\varepsilon^2),$$
so that
$$
\frac{\|f\|_{(1+\varepsilon)p}^p}{\|f\|_p^p}= 1-\varepsilon \entropy(f^p/\|f\|_p^p)+O(\varepsilon^2)
$$
for all $p>0$. The required monotonicity will therefore follow provided 
\begin{equation}\label{firstorder}
\|f\|_p^p\|f\|_{(1+\varepsilon)q}^q\leq\|f\|_q^q\|f\|_{(1+\varepsilon)p}^p
\end{equation}
for sufficiently small $\varepsilon>0$
whenever $q<p$. To see this fix $q<p$, choose $0<\varepsilon<\frac{p}{q}-1$ so that
$$
q<(1+\varepsilon)q<p<(1+\varepsilon)p,
$$
and apply H\"older's inequality (twice) to obtain
$$
\|f\|_p\leq\|f\|_q^{1-\theta}\|f\|_{(1+\varepsilon)p}^\theta, \;\;\;\frac{1}{p}=\frac{1-\theta}{q}+\frac{\theta}{(1+\varepsilon)p}
$$
and
$$
\|f\|_{(1+\varepsilon)q}\leq\|f\|_q^{1-\phi}\|f\|_{(1+\varepsilon)p}^\phi, \;\;\;\frac{1}{(1+\varepsilon)q}=\frac{1-\phi}{q}+\frac{\phi}{(1+\varepsilon)p}.
$$
Taking the appropriate powers and multiplying leads to \eqref{firstorder}. 
\end{proof}
While \eqref{subadditivityp} can be seen to fail for general functions $f$ (see Remark \ref{failure} below), it is possible to upgrade \eqref{pto0} from indicator functions to general (simple) functions by an application of the tensor power trick. 
The first observation to make is that \eqref{pto0} is tensor-power invariant. In particular, if $$B_i^{\oplus N}(x_1,\hdots,x_N):=(B_ix_1,\hdots,B_ix_N)$$ then $\BL(\mathbf{B}^{\oplus N},\mathbf{c})=\BL(\mathbf{B},\mathbf{c})^N$, as Brascamp--Lieb constants factor through critical subspaces; see \cite{BCCT2008}. Further, if $(B_i)_{*}f=f_i$ then $(B_i^{\oplus N})_{*}f^{\otimes N}=f_i^{\otimes N}$. Next we write
\begin{equation}\label{decompose}
f^{\otimes N}\sim\sum_{\ell=1}^K\phi_\ell,
\end{equation}
where the $\phi_\ell$ are nonnegative multiples of indicator functions of disjoint subsets of $\mathbb{R}^{Nd}$, and the implicit constants are independent of $N$. As $f$ is a simple function the number of terms $K$ in the decomposition \eqref{decompose} may be taken to grow linearly in $N$.
Since \eqref{pto0} holds for indicator functions,
\begin{eqnarray*}
\begin{aligned}
\|f\|_p^N=\|f^{\otimes N}\|_p&\leq c\left(\sum_{\ell=1}^K\|\phi_\ell\|_p^p\right)^{1/p}\\
&\leq c\BL(\mathbf{B}^{\oplus N},\mathbf{c})^{\frac{1}{p}-1}\left(\sum_{\ell=1}^K\prod_{i=1}^k\|(B_i^{\oplus N})_{*}\phi_\ell\|_{p_i}^{p\theta_i}\right)^{\frac{1}{p}}\\
&\leq cK^{1/p}\BL(\mathbf{B}^{\oplus N},\mathbf{c})^{\frac{1}{p}-1}\prod_{i=1}^k\|f_i^{\otimes N}\|_{p_i}^{\theta_i}\\
&\leq cK^{1/p} \BL(\mathbf{B},\mathbf{c})^{N\left(\frac{1}{p}-1\right)}\prod_{i=1}^k\|f_i\|_{p_i}^{N\theta_i},
\end{aligned}
\end{eqnarray*}
where the constant $c$ may change from line to line, but remains independent of $N$.
Taking $N$th roots and using the fact that $K^{1/N}\rightarrow 0$ establishes \eqref{pto0} for simple functions. This may be extended to general functions by a routine limiting argument. We refer to \cite{BCW2005} for a similar application of the tensor power trick in the context of Brascamp--Lieb inequalities.
\begin{remark}\label{failure}
Inequality \eqref{subadditivityp} does not hold in general, and a counterexample may be constructed in the setting of Loomis--Whitney data as follows. 
If $f=g_1\otimes\cdots\otimes g_d$ then it is straightforward to verify that
$$\entropy(f^p/\|f\|_p^p)=\sum_{i=1}^d\entropy(g_i^p/\|g_i\|_p^p)$$
and 
$$
\entropy(f_j^{p_j}/\|f_j\|_{p_j}^{p_j})=\sum_{i\not=j}\entropy(g_i^{p_j}/\|g_i\|_{p_j}^{p_j}),
$$
and so the inequality \eqref{subadditivityp} becomes
\begin{equation}\label{LW:tensorcase}
\sum_{i=1}^d\entropy(g_i^p/\|g_i\|_p^p)\leq \sum_{i=1}^d\frac{1}{d-1}\sum_{j\not=i}\entropy(g_i^{p_j}/\|g_i\|_{p_j}^{p_j}),
\end{equation}
where the exponents $p_j$ satisfy
$$
\frac{1}{d-1}\left(1-\frac{1}{p}\right)=\theta_j\left(1-\frac{1}{p_j}\right)
$$
and $\theta_1+\cdots+\theta_d=1$.
Since the $g_i$ may be chosen independently, \eqref{LW:tensorcase} reduces to
\begin{equation}\label{onedim}
\entropy(g^p/\|g\|_p^p)\leq \frac{1}{d-1}\sum_{j\not= i}\entropy(g^{p_j}/\|g\|_{p_j}^{p_j})
\end{equation}
for all $1\leq i\leq d$. If $\theta_j\leq\frac{1}{d-1}$ for all $j$ (as in the balanced case $\theta_j=\frac{1}{d}$), then $p_j\leq p$ for all $j$, and so \eqref{onedim} holds by the monotonicity of $p\mapsto \entropy(f^p/\|f\|_p^p)$ established in the proof of Proposition \ref{prop:hp}. For a counterexample we must therefore allow $\theta_j>\frac{1}{d-1}$ for some $j$. Specialising to $d=3$, the $i=3$ case of \eqref{onedim} becomes
\begin{equation}\label{LWten3}
\entropy(g^p/\|g\|_p^p)\leq\frac{1}{2}\left(\entropy(g^{p_1}/\|g\|_{p_1}^{p_1})+\entropy(g^{p_2}/\|g\|_{p_2}^{p_2})\right),
\end{equation}
and the condition on the Lebesgue exponents becomes
$$
\frac{p_1}{1-p_1}+\frac{p_2}{1-p_2}+\frac{p_3}{1-p_3}=\frac{2p}{1-p}.
$$
Since \eqref{LWten3} is independent of $p_3$, we see that if \eqref{subadditivityp} holds then necessarily \eqref{LWten3} holds whenever the exponents $0<p, p_1,p_2\leq 1$ satisfy
\begin{equation}\label{harmonic}
\frac{p}{1-p}=\frac{1}{2}\left(\frac{p_1}{1-p_1}+\frac{p_2}{1-p_2}\right).
\end{equation} 
Thus \eqref{LWten3} implies the convexity of $\frac{p}{1-p}\mapsto \entropy(g^p/\|g\|_p^p)$ for $0<p<1$. 
Suppose now that $g=1_{[0,1]}(1+\varepsilon h)$ where $h$ is supported in $[0,1]$, has mean value zero, and $\|h\|_2=1$. Observe that on $[0,1]$ we have
$$g^p=1+\varepsilon ph+\frac{p(p-1)}{2}\varepsilon^2 h^2+O(\varepsilon^3),$$
which on formally differentiating in $p$ becomes
$$
g^p\log g=\varepsilon h+\left(p-\frac{1}{2}\right)\varepsilon^2 h^2+O(\varepsilon^3).
$$
It follows that
$$
\entropy(g^p/\|g\|_p^p)=-\frac{p^2}{2}\varepsilon^2+O(\varepsilon^3),
$$
and so if $\entropy(g^p/\|g\|_p^p)$ were a convex function of $q:=\frac{p}{1-p}$ for $0<p<1$, then $q\mapsto \frac{q^2}{(q+1)^2}$ would be a concave function for $0<q<\infty$. However, a routine calculation shows that this fails for $q<\frac{1}{2}$.
\end{remark}


\section{Strict inequality in the lower bound}\label{Section:ablg-abl}
Recall from Remark \ref{Remark:deg} that in either of the degenerate situations $p=1$ or $\mathbf{c}=\theta$, the adjoint Brascamp--Lieb inequality \eqref{abli} is an identity, and so evidently $\ABL(\mathbf{B},\mathbf{c},\theta,p)=\ABLg(\mathbf{B},\mathbf{c},\theta,p)$. In this section we show that  $\ABL(\mathbf{B},\mathbf{c},\theta,p)$ and $\ABLg(\mathbf{B},\mathbf{c},\theta,p)$ differ in all other cases.
\begin{theorem}\label{ABLg<ABL}
If $0<p<1$ and $\mathbf{c}\not=\theta$ then there exists $\kappa>1$, depending only on $d, d_j, \mathbf{c}, \theta, p$, such that
$$\ABL(\mathbf{B},\mathbf{c},\theta,p)\geq \kappa\ABLg(\mathbf{B},\mathbf{c},\theta,p).$$
\end{theorem}

Our proof of Theorem \ref{ABLg<ABL} involves quantifying how  the adjoint Brascamp--Lieb functional
$$\Phi(f):=\frac{\|f\|_p}{\prod_{i=1}^k\|(B_i)_*f\|_{p_i}^{\theta_i}}$$
varies under certain carefully selected perturbations of an arbitrary (centred) gaussian input $f$. We begin with a suitably quantitative variational lemma.

\begin{lemma}\label{lemma:EL}
Suppose that $f:\mathbb{R}^d\rightarrow\mathbb{R}_{+}$ is a measurable function for which $(B_i)_*f\in L^{p_i}(\mathbb{R}^{d_i})$ for $1\leq i\leq k$. If $h:\mathbb{R}^d\rightarrow\mathbb{R}$ is a measurable function such that $h/f\in L^\infty(\mathbb{R}^d)$, then 
\begin{equation}\label{perturb}
\frac{\Phi(f+ \varepsilon h)}{\Phi(f)}=1+ \varepsilon\int h(x)\left(\frac{ f(x)^{p-1}}{\|f\|_p^p}-\sum_{i=1}^k\theta_i\frac{ f_i(B_ix)^{p_i-1}}{\|f_i\|_{p_i}^{p_i}}\right)dx+O( \varepsilon^2)
\end{equation}
as $ \varepsilon\rightarrow 0$.
Moreover, if $|h|\leq Cf$ for some positive constant $C$, then the implicit constant in the $O( \varepsilon^2)$ term may be chosen to depend only on $p, \mathbf{c}, \theta$ and $C$. 
\end{lemma}

\begin{proof}
Observe that if $h/f$ is bounded then
\begin{eqnarray*}
\begin{aligned}
\|f+ \varepsilon h\|_p^p&=\int f^p\left(1+ \varepsilon p\frac{h}{f}+O( \varepsilon^2)\right)\\&=\|f\|_p^p\left(1+ \varepsilon p\frac{\int f^{p-1}h}{\int f^p}+O( \varepsilon^2)\right),
\end{aligned}
\end{eqnarray*}
so that 
\begin{eqnarray*}
\begin{aligned}
\Phi(f+ \varepsilon h)=\Phi(f)\left(1+ \varepsilon\frac{\int f^{p-1}h}{\int f^p}+O( \varepsilon^2)\right)\prod_{i=1}^k\left(1- \varepsilon\theta_i\frac{\int f_i^{p_i-1}h_i}{\int f_i^{p_i}}+O( \varepsilon^2)\right).
\end{aligned}
\end{eqnarray*}
By \eqref{adj-def},
\begin{eqnarray*}
\begin{aligned}
\frac{\Phi(f+ \varepsilon h)}{\Phi(f)}&=1+ \varepsilon\left(\frac{\int f^{p-1}h}{\int f^p}-\sum_{i=1}^k\theta_i\frac{\int f_i^{p_i-1}h_i}{\int f_i^{p_i}}\right)+O( \varepsilon^2)\\
&=1+ \varepsilon\left(\frac{\int f^{p-1}h}{\int f^p}-\sum_{i=1}^k\theta_i\frac{\int f_i(B_i\cdot)^{p_i-1}h}{\int f_i^{p_i}}\right)+O( \varepsilon^2)\\
&=1+ \varepsilon\int h(x)\left(\frac{ f(x)^{p-1}}{\|f\|_p^p}-\sum_{i=1}^k\theta_i\frac{ f_i(B_ix)^{p_i-1}}{\|f_i\|_{p_i}^{p_i}}\right)dx+O( \varepsilon^2).
\end{aligned}
\end{eqnarray*}
The uniformity claims are a straightforward consequence of Taylor's theorem.
\end{proof}

In order to prove Theorem \ref{ABLg<ABL} we show that there exists $\eta>0$, depending only on $d, d_i, \mathbf{c}, \theta, p$, such that
$$\frac{\ABL(\mathbf{B},\mathbf{c},\theta,p)}{\ABLg(\mathbf{B},\mathbf{c},\theta,p;A)}\geq 1+\eta$$
for all $A$, where $\ABLg(\mathbf{B},\mathbf{c},\theta,p;A)$ denotes the adjoint Brascamp--Lieb constant when testing against the (centred) gaussian $f(x) = e^{-\pi \langle x, Ax \rangle}$.

By applying a linear change of variables to $\R^d$ (and adjusting $\mathbf{B}$ appropriately), we may normalize so that $A=I$.  Setting $f(x)=e^{-\pi|x|^2}$ we have $$\ABLg(\mathbf{B},\mathbf{c},\theta,p;A)=\Phi(f),$$ and so by Lemma \ref{lemma:EL} it remains to construct a suitable function $h$ for which the coefficient of $ \varepsilon$ on the right hand side of \eqref{perturb} is bounded below by a positive constant depending only on $d, d_i, \mathbf{c}, \theta, p$.
To this end let
\begin{eqnarray*}
\begin{aligned}
F(x)&=\frac{f(x)^{p-1}}{\|f\|_p^p}-\sum_{i=1}^k\theta_i\frac{f_i(B_ix)^{p_i-1}}{\|f_i\|_{p_i}^{p_i}}\\
&=p^{\frac{d}{2}}e^{-\pi(p-1)|x|^2}-\sum_{i=1}^k\theta_ip_i^{\frac{d_i}{2}}e^{-\pi(p_i-1)\langle P_i x,x\rangle} \\
&=p^{\frac{d}{2}}e^{\pi(1-p)|x|^2}\Bigl(1-p^{-\frac{d}{2}}\sum_{i=1}^k\theta_i p_i^{\frac{d_i}{2}}e^{\pi(1-p_i)\langle P_i x,x\rangle-\pi(1-p)|x|^2}\Bigr),
\end{aligned}
\end{eqnarray*}
where $P_i \colon \R^d \to \R^d$ denote the orthogonal projection
$$P_i \coloneqq B_i^*(B_i B_i^*)^{-1}B_i.$$
Recall the scaling conditions $\sum_i\theta_i=1$ and $\sum_i c_id_i=d$. Since $\theta\not=\mathbf{c}$ it follows that $\theta_j<c_j$ and hence $p_j<p$ for some $j$ by \eqref{pi-def}. With this choice of $j$ observe that
\begin{eqnarray*}
\begin{aligned}
p^{-\frac{d}{2}}\sum_{i=1}^k\theta_i p_i^{\frac{d_i}{2}}e^{\pi(1-p_i)\langle P_i x,x\rangle-\pi(1-p)|x|^2}&\geq
p^{-\frac{d}{2}}\theta_j p_j^{\frac{d_j}{2}}e^{\pi(1-p_j)\langle P_j x,x\rangle-\pi(1-p)|x|^2}\\
&\geq p^{-\frac{d}{2}}\theta_jp_j^{\frac{d_j}{2}}e^{\frac{\pi}{2}(p-p_j)|x|^2}
\end{aligned}
\end{eqnarray*}
on the cone
\begin{equation}\label{coneA}
\Gamma \coloneqq \Bigl\{x\in\mathbb{R}^d:\langle P_j x,x\rangle\geq \frac{1-\frac{1}{2}(p+p_j)}{1-p_j}|x|^2\Bigr\}.
\end{equation}
Choosing $R>0$ such that
$$p^{-\frac{d}{2}}\theta_jp_j^{\frac{d_j}{2}}e^{\frac{\pi}{2}(p-p_j)R^2}=2,$$ it follows that
$$
F(x)\leq -p^{\frac{d}{2}}e^{\pi(1-p)|x|^2}
$$
whenever $x\in\Gamma\cap B(0,R)^c$. 
If $h(x)=-f(x)1_{\Gamma\cap B(0,R)^c}(x)$ then the coefficient of $ \varepsilon$ in \eqref{perturb}
equals
\begin{eqnarray*}
\begin{aligned}
\int_{\mathbb{R}^d}h(x)F(x)dx&= -\int_{\Gamma\cap B(0,R)^c}e^{-\pi|x|^2}F(x)dx&\geq p^{\frac{d}{2}}\int_{\Gamma\cap B(0,R)^c} e^{-\pi p|x|^2}dx.
\end{aligned}
\end{eqnarray*}
Now, since $P_j$ is an orthogonal projection, the set $\Gamma\cap \mathbb{S}^{d-1}$ has measure depending only on $d, d_i, \mathbf{c}, \theta, p$ with respect to surface measure on $\mathbb{S}^{d-1}$. Hence by a use of polar coordinates,
$$
\int_{\mathbb{R}^d}h(x)F(x)dx
$$
is bounded below by a positive constant depending only on $d, d_i, \mathbf{c}, \theta, p$. Finally we note that $|f/h|, |f_i/h_i|\leq 1$ everywhere, allowing us to control the $O( \varepsilon^2)$ term in Lemma \ref{lemma:EL} suitably uniformly. The theorem now follows on taking $ \varepsilon$ to be a sufficiently small constant depending only on $d, d_i, \mathbf{c}, \theta, p$.  


\section{Strict inequality in the upper bound}\label{Section:abl-bl}
In this section we discuss the sharpness of the inequality
\begin{equation}\label{ableq}
 \ABL(\mathbf{B}, \mathbf{c}, \mathbf{\theta}, p) \leq \BL(\mathbf{B}, \mathbf{c})^{\frac{1}{p}-1}
\end{equation}
in Theorem \ref{main-cts-thm}.  It is certainly possible for equality to hold here.  For instance, for the data associated to the Loomis--Whitney and adjoint Loomis--Whitney bounds in Theorem \ref{lw-alw} (setting each of the $\Omega_j$ to be a Euclidean space), both sides of \eqref{ableq} are equal to one.  However, in situations where the only extremals of the Brascamp--Lieb inequality are gaussians, Theorem \ref{ABLg<ABL} suggests that strict inequality should occur in \eqref{ableq}.  We may formalize this intuition as follows.

\begin{definition}[Gaussian stability]\label{stability-def}  A Brascamp--Lieb datum $(\mathbf{B},\mathbf{c})$ is \emph{gaussian-stable} if $\BL(\mathbf{B},\mathbf{c})$ is finite, and whenever $f_i \colon \R^{d_i} \to \R^+$ are integrable functions such that
$$
 \int_{\R^d} \prod_{i=1}^k f_i^{c_i} \circ B_i > (1-\eps) \BL(\mathbf{B}, \mathbf{c}) \prod_{i=1}^k \left(\int_{\R^{d_i}} f_i\right)^{c_i}$$
for some $\eps>0$, then there exist (uncentered) gaussians $\tilde f_i \colon \R^{d_i} \to \R^+$ such that
$$ \| f_i - \tilde f_i \|_{L^1(\R^{d_i})} \leq o\left( \|f_i \|_{L^1(\R^{d_i})}\right)$$
where $o(X)$ denotes a quantity bounded in magnitude by $c(\eps) X$ for some function $c(\eps)$ (depending on the Brascamp--Lieb datum) that tends to zero as $\eps \to 0$.
\end{definition}

For instance, it was shown in \cite{christ-young} that the Brascamp--Lieb data corresponding to non-endpoint cases of Young's 
inequality (Example \ref{Example:Young}) are gaussian-stable.

\begin{theorem}  Let $(\mathbf{B},\mathbf{c})$ be a gaussian-stable Brascamp--Lieb datum, let $0 < p < 1$, and let $\theta_1,\dots,\theta_k>0$ sum to $1$.  Then we have strict inequality in \eqref{ableq}.
\end{theorem}

\begin{proof}  By a rescaling we may normalize so that $\BL(\mathbf{B}, \mathbf{c})=1$.

Suppose for contradiction that \eqref{ableq} holds with equality.  Let $\eps>0$ be a small parameter to be chosen later, then we can find a non-negative $f \in L^p(\R^d)$, not identically zero, such that $(B_i)_* f \in L^{p_i}(\R^{d_i})$ for all $i$ and
$$ \| f \|_{L^p(\R^d)} = (1-o(1)) \prod_{i=1}^k \| (B_i)_* f \|_{L^{p_i}(\R^{d_i})}^{\theta_i}$$
where $o()$ is as in Definition \ref{stability-def}, except that we now also permit the decay function $c()$ to depend on $p$ and $\theta_1,\dots,\theta_k$ in addition to $\mathbf{B}$ and $\mathbf{c}$.  By multiplying $f$ by a constant, we may normalize
$$ \|f\|_{L^p(\R^d)}=1$$
so that
\begin{equation}\label{bif}
\prod_{i=1}^k \| (B_i)_* f \|_{L^{p_i}(\R^{d_i})}^{\theta_i} = 1+o(1).
\end{equation}
As in Section \ref{adj-sec}, we set $f_i \coloneqq (B_i)_* f$ and $g_i \coloneqq f_i^{c_i p_i}$, then an inspection of the arguments in that section show that \eqref{lw-conseq2} must hold up to $o(1)$ errors, in the sense that
$$
\int_{\R^d} \prod_{i=1}^k g_i = (1-o(1)) \prod_{i=1}^k \|f_i\|_{L^{p_i}(\R^{d_i})}^{c_i p_i}.
$$
Applying the gaussian-stable hypothesis (with the $f_i$ replaced by $g_i^{1/c_i} = f_i^{p_i}$), we conclude that there exist uncentered gaussians $\tilde f_i$ such that
$$ \| f_i^{p_i} - \tilde f_i^{p_i} \|_{L^1(\R^{d_i})} = o\left( \| f_i \|_{L^{p_i}(\R^{d_i})}^{p_i} \right).$$
By Markov's inequality, this implies that
\begin{equation}\label{fai-0}
 \tilde f_i = (1+o(1)) f_i
\end{equation}
uniformly outside of an exceptional set $E_i \subset \R^{d_i}$ with
\begin{equation}\label{fai}
 \int_{E_i} f_i^{p_i} + \tilde f_i^{p_i} = o\left( \| f_i \|_{L^{p_i}(\R^{d_i})}^{p_i} \right)
\end{equation}
so in particular
\begin{equation}\label{fai-2}
\int_{\R^{d_i}} \tilde f_i^{p_i} = (1+o(1)) \| f_i \|_{L^{p_i}(\R^{d_i})}^{p_i}.
\end{equation}
From the adjoint Brascamp--Lieb inequality, \eqref{abli} and \eqref{bif} we have
$$ \| f 1_{B_j^{-1} E_j} \|_{L^p(\R^d)} = o\left(\prod_{i=1}^k \| (B_i)_* f \|_{L^{p_i}(\R^{d_i})}^{\theta_i}\right) = o(1)$$
for every $j=1,\dots,k$.  Thus
$$ \int_{\bigcup_{j=1}^k B_j^{-1} E_j} f^p = o( 1 ).$$
Next, define the normalized functions
\begin{align*}
G_i &\coloneqq f (g_i\circ B_i)^{-\frac{1-p}{\theta_i p}} / \|f_i\|_{L^{p_i}(\R^{d_i})}^{p_i} \\
H &\coloneqq \prod_{i=1}^k g_i\circ B_i / \|f_i\|_{L^{p_i}(\R^{d_i})}^{c_i p_i}
\end{align*}
on $\R^d$ for $i=1,\dots,k$, so by \eqref{bif} we have
$$ 
H^{1-p} \prod_{i=1}^k G_i^{\theta_i p} =(1-o(1)) f^p.
$$
An inspection of the arguments in Section \ref{adj-sec} reveals that
$$ \int_{\R^d} G_i, \int_{\R^d} H \leq 1+o(1)$$
and
$$ \int_{\R^d} H^{1-p} \prod_{i=1}^k G_i^{\theta_i p} = \int_{\R^d} f^p = 1+o(1)$$
and thus
$$ \int_{\R^d}(1-p) H + \sum_{i=1}^k \theta_i p G_i - H^{1-p} \prod_{i=1}^k G_i^{\theta_i p} = o(1);$$
the integrand is non-negative thanks to the weighted AM-GM inequality, or by Jensen's inequality applied to the exponential function. 
In view of the strict convexity of the exponential function (and homogeneity) we see that if
$$|G_i(x) - H(x)| \geq \delta H(x)
$$
for some $1\leq i\leq k$, $x \in \R^d$ and $\delta > 0$, then
\begin{equation}\label{discbound}
H(x) \leq C(p,\theta_1,\dots,\theta_k,\delta) \left((1-p) H(x) + \sum_{i=1}^k \theta_i p G_i(x) - H^{1-p}(x)\prod_{i=1}^k G_i(x)^{\theta_i p} \right)
\end{equation}
for some constant $C(p,\theta_1,\dots,\theta_k,\delta)$ depending on the indicated parameters. From this and Markov's inequality we conclude that
$$ G_i(x) = (1+o(1)) H(x)$$
for all $1\leq i\leq k$ and all $x$ outside of an exceptional set $E\subset \R^d$ with
$$\int_{E} H = o(1).$$
In particular, we have
$$ H^{1-p} \prod_{i=1}^k G_i^{\theta_i p}(x) = (1+o(1)) H(x)$$
for $x \not \in E$, and
$$ \int_E H^{1-p} \prod_{i=1}^k G_i^{\theta_i p} + H = o(1),$$
by H\"older's inequality.
Undoing the normalizations, we conclude that
$$ f(x)^p = (1+o(1)) f_*(x)^p$$
for $x \not \in E$ and
\begin{equation}\label{ef}
 \int_E f^p + f_*^p = o(1)
\end{equation}
where
$$ f_* \coloneqq \prod_{i=1}^k (f_i\circ B_i)^{c_i p_i/p} /  \|f_i\|_{L^{p_i}(\R^{d_i})}^{(c_i p_i - \theta_ip)/p}.$$ 
If we let $\tilde E \coloneqq E \cup \bigcup_{i=1}^k B_i^{-1} E_i$, we conclude that
\begin{equation}\label{ftf}
 f(x) = (1+o(1)) \tilde f(x)
\end{equation}
for $x \not \in \tilde E$ and
$$ \int_{\tilde E} f^p = o(1),$$
where $\tilde f$ is the gaussian function
$$ \tilde f \coloneqq \prod_{i=1}^k (\tilde f_i\circ B_i)^{c_i p_i/p} /  \|f_i\|_{L^{p_i}(\R^{d_i})}^{(c_i p_i - \theta_ip)/p}.$$
Also, from the Brascamp--Lieb inequality and \eqref{fai}, \eqref{fai-2} we have
$$ \int_{B_i^{-1} E_i} \tilde f^p = o(1)$$
for all $i=1,\dots,k$, while from \eqref{ef}, \eqref{fai-0} we have
$$ \int_{E \backslash \bigcup_{i=1}^k B_i^{-1} E_i} \tilde f^p = o( 1 ),$$
and thus
$$ \int_{\tilde E} \tilde f^p = o(1).$$
In particular we have
\begin{equation}\label{rad}
\begin{split}
 \int_{\R^d} \tilde f^p &= \int_{\R^d \backslash \tilde E} \tilde f^p + o(1) \\
&= (1+o(1)) \int_{\R^d \backslash \tilde E} f^p + o(1)\\
&= 1+o(1) 
\end{split}
\end{equation}
and hence
$$ \int_{\tilde E} \tilde f^p = o\left(\int_{\R^d} \tilde f^p\right).$$
Next let $q>1$ and choose $0<\theta<1$ such that $1=\frac{\theta}{p}+\frac{1-\theta}{q}$, so that by H\"older's inequality,
$$
 \int_{\tilde E} \tilde f\leq\|\tilde f\|_{L^p(\tilde E)}^\theta\|\tilde f\|_{L^q(\tilde E)}^{1-\theta}=o\left(\|\tilde f\|_{L^p(\mathbb{R}^d)}^\theta\|\tilde f\|_{L^q(\mathbb{R}^d)}^{1-\theta}\right).
$$
Since $\tilde f$ is a gaussian it satisfies a reverse H\"older inequality on $\mathbb{R}^d$, and so it follows that
$$ \int_{\tilde E} \tilde f = o\left(\int_{\R^d} \tilde f\right).$$
Pushing forward by $B_i$, we conclude that
$$ \int_{\R^{d_i}} (B_i)_* (1_{\tilde E} \tilde f) = o\left(\int_{\R^{d_i}} (B_i)_* \tilde f\right).$$
By Markov's inequality it follows that
$$ (B_i)_* (1_{\tilde E} \tilde f) = o( (B_i)_* (\tilde f) )$$
outside of an exceptional set $E'_i \subset \R^{d_i}$ with
$$ \int_{E'_i} (B_i)_* \tilde f = o\left(\int_{\R^{d_i}} (B_i)_* \tilde f\right).$$
Note that we can assume that $E'_i$ contains $E_i$.  Since $(B_i)_* \tilde f$ is also gaussian, we conclude from a further appeal to H\"older's inequality that
\begin{equation}\label{eip}
 \int_{E'_i} ((B_i)_* \tilde f)^{p_i} = o\left(\int_{\R^{d_i}} ((B_i)_* \tilde f)^{p_i}\right).
\end{equation}
Outside of $E'_i$, we have
$$ (B_i)_* ((1-1_{\tilde E}) \tilde f) = (1-o(1)) (B_i)_* \tilde f$$
and hence by \eqref{ftf},
$$ f_i \geq (B_i)_* ((1-1_{\tilde E}) f) = (1-o(1)) (B_i)_* \tilde f.$$
In particular
$$ \int_{\R^{d_i} \backslash E'_i} ((B_i)_* \tilde f)^{p_i} \leq (1+o(1)) \|f_i\|_{L^{p_i}(\R^{d_i})}^{p_i}$$
and thus by \eqref{eip},
$$ \int_{\R^{d_i}} ((B_i)_* \tilde f)^{p_i} \leq (1+o(1)) \|f_i\|_{L^{p_i}(\R^{d_i})}^{p_i}.$$
Comparing this inequality with \eqref{rad} and Definition \ref{cts-adj}, we conclude that
$$ \ABLg( \mathbf{B}, \mathbf{c}, \mathbf{\theta}, p) \geq 1-o(1)$$
and thus 
$$ \ABLg( \mathbf{B}, \mathbf{c}, \mathbf{\theta}, p) \geq (1-o(1)) \ABL( \mathbf{B}, \mathbf{c}, \mathbf{\theta}, p).$$
This contradicts Theorem \ref{ABLg<ABL}.
\end{proof}


\section{Lower bounds for tomographic transforms}\label{Section:X-ray}
A simple yet effective application of the adjoint Brascamp--Lieb inequality is to lower bounds for various tomographic transforms between Lebesgue spaces. These bounds lie below $L^1$, in contrast with the more familiar upper bounds (see \cite{Christ1984}) that necessarily lie above $L^1$. 

We begin with the classical X-ray transform, which is governed by the adjoint Loomis--Whitney inequality from Section \ref{Section:loomis}.
Let $d \geq 2$. For a nonnegative measurable function $f \colon \R ^d\rightarrow \R^+$ its \emph{X-ray transform} $Xf \colon \mathcal{M}_{1,d} \to \R^+$ is given by the formula
\begin{equation}\label{Xray-def}
Xf(\omega,v) \coloneqq \int_{\R }f(v+t\omega)dt,
\end{equation}
where the variables $\omega\in\mathbb{S}^{d-1}$ and $v\in\langle\omega\rangle^\perp$ form the natural parametrization of the Grassmannian manifold $\mathcal{M}_{1,d}$ of lines in $\R^d$, endowed with the obvious measure
\begin{equation}\label{obvious}
 \int_{\mathcal{M}_{1,d}} F(\omega,v) \coloneqq \int_{\mathbb{S}^{d-1}} \left(\int_{\langle \omega \rangle^\perp} F(\omega,v)\ dv\right) d\sigma(\omega)
\end{equation}
where the inner integral on the right-hand side is with respect to Lebesgue measure $dv$ on $\langle \omega \rangle^\perp$, and $d\sigma$ denotes normalized Lebesgue measure on $\mathbb{S}^{d-1}$. This normalization is chosen so that $\|Xf\|_{L^1(\mathcal{M}_{1,d})}=\|f\|_{L^1(\mathbb{R}^d)}$.
\begin{theorem}[Lower bound for the $X$-ray transform]\label{theoremX}
Suppose that $0<p, q\leq 1$ and $d \geq 2$. Then there exists a positive constant $C = C_{p,q,d}$ such that
\begin{equation}\label{lowerX}
\|Xf\|_{L^q({\mathcal M}_{1,d})} \geq C\|f\|_{L^p(\R^d)}
\end{equation}
for all nonnegative measurable functions $f \colon \R^d \to \R^+$
if and only if
\begin{equation}\label{Xcond}
\frac{1}{d}\left(1-\frac{1}{q}\right)=\frac{1}{d-1}\left(1-\frac{1}{p}\right).
\end{equation}
Moreover, in all such cases we may take $C=1$.
\end{theorem}

\begin{proof}
That \eqref{Xcond} is necessary for \eqref{lowerX} follows by a routine scaling argument, as is familiar from the analysis of upper bounds in \cite{Christ1984}.

For $\omega\in\mathbb{S}^{d-1}$ let $P_\omega$ denote the orthogonal projection of $\R ^{d}$ onto $\langle\omega\rangle^\perp$. 
If $\{\omega_1,\hdots,\omega_d\}$ is an orthonormal basis of $\mathbb{R}^d$ then by the Loomis--Whitney inequality (Theorem \ref{lw-alw}(i); see also Example \ref{Example:LW}) we have
\begin{equation}\label{LWrot}
\int_{\R ^{d}}\prod_{i=1}^df_i(P_{\omega_i}x)^{\frac{1}{d-1}}dx\leq \prod_{i=1}^d\left(\int_{\langle\omega_i\rangle^\perp}f_i\right)^{\frac{1}{d-1}}
\end{equation}
for any measurable $f_i \colon \langle \omega_i \rangle^\perp \to \R^+$, 
for which (as in Theorem \ref{lw-alw}(ii)) we have the adjoint inequality
\begin{equation}\label{adjrot}
\|f\|_{L^p(\R ^d)}\leq  \prod_{i=1}^d\|(P_{\omega_i})_*f\|_{L^q(\langle\omega_i\rangle^\perp)}^{\frac{1}{d}},
\end{equation}
thanks to \eqref{Xcond}. But from \eqref{Xray-def}, the Fubini--Tonelli theorem, and a change of variables we have
\begin{equation}\label{FT}
\|(P_{\omega_i})_*f\|_{L^q(\langle\omega_i\rangle^\perp)}=\|Xf(\omega_i,\cdot)\|_{L^q(\langle\omega_i\rangle^\perp)},
\end{equation}
and thus
\begin{equation}\label{adjrot1}
\|f\|_{L^p(\R ^d)}^{q}\leq \prod_{i=1}^d\|Xf(\omega_i,\cdot)\|_{L^q(\langle\omega_i\rangle^\perp)}^{\frac{q}{d}},
\end{equation}
which by the AM-GM inequality implies
\begin{equation}\label{adjrot2}
\|f\|_{L^p(\R ^d)}^{q}\leq \frac{1}{d}\sum_{i=1}^d\|Xf(\omega_i,\cdot)\|_{L^q(\langle\omega_i\rangle^\perp)}^{q}.
\end{equation}
The inequality \eqref{lowerX} now follows on averaging over all such bases $\{\omega_1,\hdots,\omega_d\}$ and then applying \eqref{obvious}.
\end{proof}
\begin{remark}
As should be evident from our proof, the $L^q({\mathcal M}_{1,d})$ norm in \eqref{lowerX} may be replaced by the mixed norm $L^r_\omega L^q_v$ for any $r>0$. A similar remark applies to all of the results in this section.
\end{remark}

\begin{remark}\label{Remark:constant}
Since we may take $C=1$ it is possible to take a limit in \eqref{lowerX} (raised to the $p^{\operatorname{th}}$ power) as $p\rightarrow 0$, obtaining
\begin{equation}\label{iso}
|\Omega|^{\frac{d-1}{d}}\leq\int_{\mathbb{S}^{d-1}}|P_{\omega}\Omega|d\sigma(\omega)
\end{equation}
for measurable subsets $\Omega\subseteq\R^d$.  This can be viewed as an averaged form of the classical Loomis--Whitney inequality \cite{Loomis-Whitney}
$$
|\Omega|^{\frac{d-1}{d}}\leq \prod_{i=1}^d |P_{\omega_i} \Omega|^{\frac{1}{d}}
$$ 
for $\{\omega_1,\dots,\omega_d\}$ an orthonormal basis of $\R^d$.  We refer to \cite[Section 7]{Milman2022} for similar basis-averaging arguments. When $\Omega$ is convex, there are a number of \emph{reverse Loomis--Whitney inequalities} available, starting with Meyer's inequality \cite{Meyer1988}
$$
|\Omega|^{\frac{d-1}{d}}\geq \frac{((d-1)!)^{\frac{1}{d}}}{d^{\frac{d-1}{d}}} \prod_{i=1}^d |\Omega \cap \langle \omega_i\rangle^\perp|^{\frac{1}{d}}
$$ 
for an orthonormal basis $\{\omega_1,\dots,\omega_d\}$; see for instance \cite{ABBC2021} for recent work in this direction.  We do not know if these inequalities can similarly be viewed as $p \to 0$ limits of some $L^p$ inequality.
\end{remark}

\begin{remark}  Theorem \ref{theoremX} breaks down when $f$ is not assumed to be non-negative.  Suppose for instance that $f$ is chosen so that the Fourier transform $\hat f$ is a bump function adapted to a cylindrical region $B^{d-1}(0,\delta) \times [1,2]$ for some small $\delta > 0$.  Then $f$ is comparable to $\delta^{d-1}$ on a disk $B^{d-1}(0,c/\delta) \times [-c,c]$ for some small $c>0$, hence
$$ \|f\|_{L^p(\R^d)} \gtrsim \delta^{d-1-\frac{d-1}{p}}$$
for any $p>0$.  On the other hand, observe that $Xf(\omega,v)$ vanishes unless the hyperplane $\omega^\perp$ passes through the support
$B^{d-1}(0,\delta) \times [1,2]$ of $\hat f$, which only occurs for $\omega$ in a strip of measure $O(\delta)$, and for such $\omega$, $Xf(\omega,\cdot)$ is concentrated in a box of dimensions $1/\delta \times \dots \times 1/\delta \times 1$ in $\omega^\perp$ and has amplitude $O(\delta^{d-2})$, hence
$$ \|Xf\|_{L^q(\R^d)} \lesssim \delta^{d-2-\frac{d-3}{q}}.$$
Sending $\delta\to 0$, we conclude that the inequality \eqref{lowerX} can only hold uniformly in $\delta$ if
$$ d-2 - \frac{d-3}{q} \leq d-1 - \frac{d-1}{p},$$
which is only compatible with \eqref{Xcond} when $p \geq \frac{d+1}{2}$, which is absurd since $0 < p \leq 1$.

This counterexample shows that Theorem \ref{theoremX} does not follow from the existence of a bounded linear left inverse of $X$ from the non-negative functions of $L^q(\mathcal{M}_{1,d})$ to the non-negative functions of $L^p(\R^d)$, since such an inverse would also extend to signed functions and thus contradict the above counterexample.  
\end{remark}
If we are prepared to accept a slightly smaller positive constant $C$, then Theorem \ref{theoremX} may be extended to quite general \emph{restricted X-ray transforms}. This involves replacing the uniform measure on $\mathbb{S}^{d-1}$ with a more general positive finite measure $\mu$, and this is naturally extend to a measure $\nu$ on $\mathcal{M}_{1,d}$ by
\begin{equation}\label{obviousE}
\int_{\mathcal{M}_{1,d}}Fd\nu \coloneqq \int_{\mathbb{S}^{d-1}} \left(\int_{\langle \omega \rangle^\perp} F(\omega,v)\ dv\right) d\mu(\omega).
\end{equation}
\begin{theorem}[Lower bounds for restricted $X$-ray transforms]\label{theoremXE}
Suppose $d\geq 2$ and that $0<p,q\leq 1$ satisfy \eqref{Xcond}. Then
\begin{equation}\label{lowerXE}
\|Xf\|_{L^q(d\nu)} \geq C(\mu)\|f\|_{L^p(\R^d)}
\end{equation}
for all nonnegative measurable functions $f \colon \R^d \to \R^+$, where
$$
C(\mu)\coloneqq\left(\int_{(\mathbb{S}^{d-1})^d}|\omega_1\wedge\cdots\wedge\omega_d|^{\frac{dq}{d-1}\left(\frac{1}{p}-1\right)} d\mu(\omega_1)\cdots d\mu(\omega_d)\right)^{\frac{1}{dq}}.
$$
\end{theorem}

\begin{proof}
If $\omega_1,\hdots,\omega_d\in\mathbb{S}^{d-1}$ are linearly independent, then by the Loomis--Whitney inequality (Theorem \ref{lw-alw}(i)) combined with suitable linear changes of variables (see for example \cite[Appendix A]{BB2010}) we have 
\begin{equation}\label{LWinv}
\int_{\R ^{d-1}}\prod_{j=1}^df_i(P_{\omega_i}x)^{\frac{1}{d-1}}dx\leq |\omega_1\wedge\cdots\wedge\omega_d|^{-\frac{1}{d-1}}\prod_{i=1}^d\left(\int_{\langle\omega_i\rangle^\perp}f_i\right)^{\frac{1}{d-1}}
\end{equation}
for any measurable $f_i \colon \langle \omega_i \rangle^\perp \to \R^+$. We refer to \cite[Section 9.2]{CHV2023} for an alternative proof of this affine-invariant form of the Loomis--Whitney inequality. For this inequality (as in Theorem \ref{lw-alw}(ii)) we have the adjoint inequality
\begin{equation}\label{adjomega}
\|f\|_{L^p(\R ^d)}\leq  |\omega_1\wedge\cdots\wedge\omega_d|^{-\frac{1}{d-1}\left(\frac{1}{p}-1\right)}\prod_{i=1}^d\|(P_{\omega_i})_*f\|_{L^q(\langle\omega_i\rangle^\perp)}^{\frac{1}{d}},
\end{equation}
thanks to \eqref{Xcond}. Applying \eqref{FT} it follows that
\begin{equation}\label{adjomega1}
 |\omega_1\wedge\cdots\wedge\omega_d|^{\frac{dq}{d-1}\left(\frac{1}{p}-1\right)}\|f\|_{L^p(\R ^d)}^{dq}\leq \prod_{i=1}^d\|Xf(\omega_i,\cdot)\|_{L^q(\langle\omega_i\rangle^\perp)}^q.
\end{equation}
The inequality \eqref{lowerXE} now follows on integrating in $\omega_i$ with respect to the measure $\mu$ for each $i$, and then applying \eqref{obviousE}.
\end{proof}
\begin{remark}
Evidently Theorem \ref{theoremXE} only has content for measures $\mu$ for which $C(\mu)$ is nonzero, a condition that is manifestly independent of the choice of $p\not=1$. This happens if and only if the support of $\mu$ is not contained in any great sphere. This geometric condition is easily seen to be best-possible for $p\not=1$. For example, if $\mu$ is supported in the great sphere perpendicular to $e_d$, and $f$ is the indicator function of the cylinder $\{x=(x',x_d)\in\mathbb{R}^d: x'\in\mathbb{R}^{d-1}, |x'|\leq 1, |x_d|\leq\delta\}$ then $Xf(\omega,v)\leq 1$ for $\omega\in\supp(\mu)$ and $v$ belonging to a set of measure approximately $\delta$, and vanishes otherwise. Hence $\|Xf\|_{L^q(d\nu)}\lesssim \delta^{\frac{1}{q}}$ and $\|f\|_p\sim\delta^{\frac{1}{p}}$. Thus if $\|f\|_{L^p(\mathbb{R}^d)}\lesssim \|Xf\|_{L^q(d\nu)}$ then necessarily $p=q$, which forces $p=1$ by \eqref{Xcond}. We remark that the condition \eqref{Xcond} continues to be best-possible in Theorem \ref{theoremXE} as the standard scaling argument is not sensitive to the choice of measure $\mu$.
\end{remark}
Theorem \ref{theoremX} may be generalized to the \emph{$k$-plane transform}
$$
T_{k,d}f(\pi,y) \coloneqq \int_\pi f(x+y)d\lambda_\pi(x),
$$
where the Grassmannian manifold $\mathcal{M}_{k,d}$ of affine $k$-planes is parametrized by a $k$-dimensional subspace $\pi$ and an element $y\in\pi^\perp$. 
\begin{theorem}[Lower bound for the $k$-plane transform]
\label{theoremTk}
Suppose that $0<p, q\leq 1$, $d \geq 2$ and $1\leq k\leq d-1$. Then there exists a positive constant $C = C_{p,q,d,k}$ such that
\begin{equation}\label{lowerTk}
\|T_{k,d}f\|_{L^q(\mathcal{M}_{k,d})}\geq C\|f\|_p
\end{equation}
for all nonnegative measurable functions $f \colon \R^d \to \R^+$
if and only if
\begin{equation}\label{Tkcond}
\frac{1}{d}\left(1-\frac{1}{q}\right)=\frac{1}{d-k}\left(1-\frac{1}{p}\right).
\end{equation}
Moreover, in all such cases we may take $C=1$.
\end{theorem}
\begin{proof}
That \eqref{lowerTk} implies \eqref{Tkcond} follows by scaling.

For a $k$-dimensional subspace $\pi$ let $P_{\pi}$ be the orthogonal projection of $\R ^d$ onto $\pi^\perp$. Next we let $(E_i)_{i=1}^n$ be an enumeration of the $k$-dimensional coordinate subspaces of $\R^d$, so that $E_i=\langle e_j\rangle_{j\in S_i}$ where $S_i$ is a subset of $\{1,\hdots,d\}$ of cardinality $k$. Evidently $n=\binom{d}{d-k}$ and $$\#\{i:j\in S_i\}=\binom{d-1}{d-k-1}$$ for each $1\leq j\leq d$. Hence if $c_i=\binom{d-1}{d-k-1}^{-1}$ for each $i$,
$$((P_{E_i})_{i=1}^n, (c_i)_{i=1}^n)$$ is a Brascamp--Lieb datum of Finner type (see Remark \ref{finner-rem}), and so 
\begin{equation}\label{finn}
\BL((P_{E_i})_{i=1}^n, (c_i)_{i=1}^n)=1.
\end{equation}
By rotation invariance (or by selecting a different orthonormal basis of $\mathbb{R}^d$) it follows that
$$
\BL((P_{\rho E_i})_{i=1}^n, (c_i)_{i=1}^n)=1
$$
for all rotations $\rho$ of $\mathbb{R}^d$.
Taking suitable adjoints of these Brascamp--Lieb inequalities and arguing as in the proof of Theorem \ref{theoremX} we have
\begin{eqnarray}\label{adjrho}
\begin{aligned}
\|f\|_{L^p(\R ^d)}^q&\leq\prod_{i=1}^n\|(P_{\rho E_i})_*f\|_{L^q((\rho E_i)^\perp)}^{\frac{q}{n}}\\
&=\prod_{i=1}^n\|T_{k,d}f(\rho E_i,\cdot)\|_{L^q((\rho E_i)^\perp)}^{\frac{q}{n}}\\
&\leq\frac{1}{n}\sum_{i=1}^n\|T_{k,d}f(\rho E_i,\cdot)\|_{L^q((\rho E_i)^\perp)}^{q}.
\end{aligned}
\end{eqnarray}
Here we have used the fact that $$\binom{d-1}{d-k-1}^{-1}\left(1-\frac{1}{p}\right)=\binom{d}{d-k}^{-1}\left(1-\frac{1}{q}\right)$$ by \eqref{Tkcond}. The inequality \eqref{lowerTk} follows from \eqref{adjrho} by averaging in $\rho$.
\end{proof}
\begin{remark}
Theorem \ref{theoremTk} is somewhat similar in spirit to certain slicing problems in convex geometry, such as Bourgain's slicing problem, which may also be interpreted as lower bounds on Radon-like transforms. We refer to \cite{KRZ2023} for further contextual discussion and some recent results. 
\end{remark}
\begin{remark}
An extension of Theorem \ref{theoremXE} to the $k$-plane transform is less apparent as the Brascamp--Lieb constant 
\begin{equation}\label{mysteryconstant}
\BL((P_{\pi_i})_{i=1}^n, (c_i)_{i=1}^n)
\end{equation}  
is rather more mysterious in the intermediate range $1<k<d-1$; see
the proof of Theorem \ref{theoremTk}. However, it may be possible to sidestep this by adapting the proof of the forthcoming Theorem \ref{theoremTkmon}, which uses only X-ray transform estimates. For some recent discussion of restricted $k$-plane transforms see \cite{DMV2023}.
\end{remark}
That we may take $C=1$ in Theorem \ref{theoremTk} hints at the following stronger result.
\begin{theorem}[Monotonicity of $k$-plane transform norms]\label{theoremTkmon}
Suppose $d\geq 2$, $0<p\leq 1$ and that the exponent $p_k$ is given by
\begin{equation}\label{Tkkcond}
\frac{1}{d}\left(1-\frac{1}{p_k}\right)=\frac{1}{d-k}\left(1-\frac{1}{p}\right) 
\end{equation}
for each $1\leq k\leq d-1$.
Then $\|T_{k,d}f\|_{L^{p_k}(\mathcal{M}_{k,d})}$ is nondecreasing in $k$.
\end{theorem}
\begin{proof}
We are required to show that 
\begin{equation}\label{quickmon}
\|T_{k-1,d}f\|_{L^{p_{k-1}}(\mathcal{M}_{k-1,d})}\leq \|T_{k,d}f\|_{L^{p_k}(\mathcal{M}_{k,d})}
\end{equation} 
for all $1\leq k\leq d-1$. Interpreting $T_{0,d}$ as the identity, the case $k=1$ (and thus $d=2$) follows from Theorem \ref{theoremX}. We may therefore suppose that $d\geq 3$ and $2\leq k\leq d-1$.
Let $\pi$ be a $(k-1)$-dimensional subspace of $\mathbb{R}^d$. Applying the $(d-k+1)$-dimensional X-ray transform to
$
T_{k-1,d}f(\pi,\cdot)
$
on the euclidean space $\pi^\perp$
we have
\begin{equation}\label{k-1 to k}
X\left(T_{k-1,d}f(\pi,\cdot)\right)(\omega,v)=T_{k,d}f(\pi\oplus\langle\omega\rangle,v)
\end{equation}
where $\omega$ is a unit vector in $\pi^\perp$ and $v\in(\pi\oplus\langle\omega\rangle)^\perp$, the orthogonal complement of $\langle\omega\rangle$ in $\pi^\perp$. Applying Theorem \ref{theoremX} and using \eqref{Tkkcond} we have
$$
\|T_{k-1,d}f(\pi,\cdot)\|_{L^{p_{k-1}}(\pi^\perp)}^{p_k}\leq \int\int_{(\pi\oplus\langle\omega\rangle)^\perp}T_{k,d}f(\pi\oplus\langle\omega\rangle,v)^{p_k}dvd\sigma_{\pi^\perp}(\omega),
$$
where $d\sigma_{\pi^\perp}$ denotes normalized surface measure on the unit sphere of $\pi^\perp$. Integrating in $\pi$ and taking roots we obtain
$$
\|T_{k-1,d}f\|_{L^{p_k}_\pi L^{p_{k-1}}_y}\leq  \|T_{k,d}f\|_{L^{p_k}(\mathcal{M}_{k,d})}.
$$
The claimed inequality \eqref{quickmon} now follows by H\"older's inequality since $p_{k-1}\leq p_k$.
\end{proof}
As a corollary of Theorem \ref{theoremTkmon} we have the following sharp entropy inequalities.
\begin{corollary}[Monotonicity of $k$-plane transform entropies]\label{corTkmon}
 The sequence of normalized entropies
\begin{equation}\label{normalentropy}
\frac{1}{d-k}\entropy(T_{k,d}f)
\end{equation} 
is nondecreasing, and in particular,
\begin{equation}\label{entropyTk}
\entropy(T_{k,d}f)\geq \left(\frac{d-k}{d}\right)\entropy(f)
\end{equation}
for all $1\leq k\leq d-1$. 
\end{corollary}
Corollary \ref{corTkmon} follows as a certain $p\rightarrow 1$ limit of \eqref{quickmon}. A very similar argument may be found in Section \ref{Section:entropy}, and so we leave the details to the interested reader.
\begin{remark}
The inequality \eqref{entropyTk} is easily seen to hold with equality when $f(x)=e^{-\pi |x|^2}$. Indeed for these extremising inputs the sequence of normalized entropies \eqref{normalentropy} is identically $1$.
Corollary \ref{corTkmon} may also be proved directly from entropy Brascamp--Lieb inequalities by suitably adapting our arguments.
\end{remark}
The arguments of this section are also effective when $p>1$ via the \emph{reverse} adjoint Brascamp--Lieb inequality of Theorem \ref{Theorem:reverse}. For the X-ray transform we have the following.
\begin{theorem}\label{theoremXX}
Suppose $d\geq 2$, $1< p<\infty$ and $0<q< 1$. If $1< r<\infty$ satisfies
\begin{equation}\label{XXcond}
\left(\frac{1}{q}-\frac{1}{p}\right)\left(1-\frac{1}{r}\right)=\frac{1}{d-1}\left(1-\frac{1}{p}\right)\left(\frac{1}{q}-1\right),
\end{equation}
then there exists a constant $C>0$ such that
\begin{equation}\label{XX}
C\|Xf\|_{L^\infty_\omega L^r_{v}}^{\frac{1}{q}-\frac{1}{p}}
\leq \|f\|_p^{\frac{1}{q}-1}\|Xf\|_{L_{\omega,v}^{q}}^{1-\frac{1}{p}}
\end{equation}
for all nonnegative functions $f:\mathbb{R}^d\rightarrow\mathbb{R}^+$.
Morever, if the condition \eqref{XXcond} is not satisfied then \eqref{XX} fails for all positive $C$.
\end{theorem}

\begin{proof}
That the condition \eqref{XXcond} is necessary follows by the usual scaling argument.

Applying Theorem \ref{Theorem:reverse} to the affine-invariant Loomis--Whitney inequality \eqref{LWinv} we have,
\begin{equation}\label{revready}
 |\omega_1\wedge\cdots\wedge\omega_d|^{\frac{1}{d-1}(1-\frac{1}{p})}\|(P_{\omega_d})_*f\|_{p_d}^{\theta_d}\leq\|f\|_p\prod_{i=1}^{d-1}\|(P_{\omega_i})_*f\|_{p_i}^{-\theta_i},
\end{equation}
where $-\theta_1,\hdots,-\theta_{d-1},\theta_d>0$, $0<p_1,\hdots, p_{d-1}\leq 1$ and $p,p_d\geq 1$ satisfy
$\theta_1+\cdots +\theta_d=1$
and 
$\theta_i\left(1-\frac{1}{p_i}\right)=\frac{1}{d-1}\left(1-\frac{1}{p}\right)$
for each $j$. 
Setting $p_1=\cdots =p_{d-1}=q$ and $$-\theta_1=\cdots=-\theta_{d-1}=\frac{1}{d-1}\left(1-\frac{1}{p}\right)/\left(\frac{1}{q}-1\right)$$ it follows that 
$$
\theta_d=\left(\frac{1}{q}-\frac{1}{p}\right)/\left(\frac{1}{q}-1\right).
$$
Furthermore, we have $p_d=r$ by \eqref{XXcond}. Therefore,
\eqref{revready} becomes
\begin{equation}\label{revinv}
|\omega_1\wedge\cdots\wedge\omega_d|^{\frac{1}{d-1}(1-\frac{1}{p})}\|(P_{\omega_d})_*f\|_{r}^{\left(\frac{1}{q}-\frac{1}{p}\right)/\left(\frac{1}{q}-1\right)}\leq \|f\|_p\prod_{i=1}^{d-1}\|(P_{\omega_i})_*f\|_{q}^{\frac{1}{d-1}\left(1-\frac{1}{p}\right)/\left(\frac{1}{q}-1\right)}.
\end{equation}
Raising this to a suitable power
and integrating with respect to $\omega_1,\hdots,\omega_{d-1}$ we obtain
\begin{equation}\label{nearly}
C\|Xf\|_{L^r_{v}(\langle\omega_d\rangle^\perp)}^{\frac{1}{q}-\frac{1}{p}}
\leq \|f\|_p^{\frac{1}{q}-1}\|Xf\|_{L_{\omega,v}^{q}}^{1-\frac{1}{p}},
\end{equation}
where 
\begin{equation}\label{theconstant}
C=\left(\int_{(\mathbb{S}^{d-1})^{d-1}}|\omega_1\wedge\cdots\wedge\omega_d|^{1-q}d\sigma(\omega_1)\cdots d\sigma(\omega_{d-1})\right)^{\frac{1}{d-1}\left(1-\frac{1}{p}\right)\frac{1}{q}}.
\end{equation}
By rotation-invariance this positive expression is independent of $\omega_d$, and so \eqref{XX} follows from \eqref{nearly} on taking a supremum in $\omega_d$. 
\end{proof}
\begin{remark}
The inequality \eqref{XX} captures a certain ``$L^p$-self-improving" property of $Xf$, as the simple two-dimensional example 
$$
\|Xf\|_{L^\infty_\omega L_v^{\frac{4}{3}}}^2\lesssim\|f\|_2\|Xf\|_{L^{\frac{2}{3}}_{\omega,v}}
$$
illustrates.
On a superficial level these inequalities are reminiscent of Nash's inequality
$$
\|f\|_{L^2(\mathbb{R}^d)}^{1+\frac{2}{d}}\lesssim \|f\|_{L^1(\mathbb{R}^d)}^{\frac{2}{d}}\|(-\Delta)^{\frac{1}{2}}f\|_{L^2(\mathbb{R})},
$$
which may be written in terms of the X-ray transform as
$$
\|(-\Delta_v)^{\frac{1}{4}}Xf\|_{L^2_{\omega,v}}^{1+\frac{1}{d}}\lesssim\|f\|_{L^1(\mathbb{R}^d)}^{\frac{2}{d}}\||(-\Delta_v)^{\frac{3}{4}}Xf\|_{L^2_{\omega,v}}.
$$
\end{remark}

\begin{remark}
The constant $C$ in the statement of Theorem \ref{theoremXX} may be evaluated somewhat explicitly. From the proof of Theorem \ref{theoremXX} we have
\begin{equation}\label{constant}
C^{(d-1)p'q}=\int_{(\mathbb{S}^{d-1})^{d}}|\omega_1\wedge\cdots\wedge\omega_d|^{a}d\sigma(\omega_1)\cdots d\sigma(\omega_{d-1})
\end{equation}
where $a=1-q$. Consider the gaussian integral
\begin{equation}\label{gauss-int}
 \int_{(\R^d)^d} |x_1 \wedge \dots \wedge x_d|^a e^{-\pi(|x_1|^2 + \dots + |x_d|^2)}\ dx_1 \dots dx_d.
\end{equation}
On the one hand, by polar coordinates in each of the variables $x_1,\dots,x_d \in \R^{d}$, together with \eqref{constant} and the standard identity
\begin{equation}\label{rp}
 \int_0^\infty r^{\alpha-1} e^{-\pi r^2}\ dr = \frac{\Gamma(\frac{\alpha}{2})}{2\pi^{\frac{\alpha}{2}}} 
\end{equation}
for any $\alpha>0$, we see that the integral \eqref{gauss-int} evaluates to
$$ C^{(d-1)p'q} \left( \frac{\Gamma(\frac{d+a}{2})}{\pi^{\frac{a}{2}}\Gamma(\frac{d}{2})}\right)^d.$$
On the other hand, we have the base times height formula
$$|x_1\wedge\cdots\wedge x_d| = \prod_{\ell=0}^{d-1} \mathrm{dist}( x_{\ell+1}, \mathrm{span}( x_1,\dots,x_\ell) ).$$
For any $\ell$-dimensional subspace $V$ of $\R^d$, a polar decomposition in $V$ and $V^\perp$ separately using \eqref{rp} reveals that
$$ \frac{\int_{\R^d} \mathrm{dist}(x_{\ell+1}, V)^a e^{-\pi |x_{\ell+1}|^2}\ dx_{\ell+1}}{\int_{\R^d} e^{-\pi |x_{\ell+1}|^2}\ dx_{\ell+1}}
= \pi^{-\frac{a}{2}} \frac{\Gamma(\frac{d-\ell+a}{2})}{\Gamma(\frac{d-\ell}{2})}.$$
As is well known, the denominator on the left-hand side is equal to $1$.  Setting $V = \mathrm{span}( x_1,\dots,x_\ell)$ and using the Fubini--Tonelli theorem repeatedly (noting that $x_1,\dots,x_d$ are linearly independent outside of a set of measure zero), we conclude that \eqref{gauss-int} also evaluates to
$$ \prod_{\ell=0}^{d-1} \pi^{-\frac{a}{2}} \frac{\Gamma(\frac{d-\ell+a}{2})}{\Gamma(\frac{d-\ell}{2})}.$$
Comparing the two formulae, we conclude that
\begin{equation}\label{constant-2}
 C = \left( \prod_{\ell=0}^{d-1}  \frac{\Gamma(\frac{d}{2})\Gamma(\frac{d-\ell+1-q}{2})}{\Gamma(\frac{d+1-q}{2}) \Gamma(\frac{d-\ell}{2})} \right)^{\frac{1}{d-1}\left(1-\frac{1}{p}\right)\frac{1}{q}}.
\end{equation}
This formula is implicit in \cite{Miles1971}, where expressions for the (ostensibly integer) moments of volumes of various random simplices in $\mathbb{R}^d$ are provided; see Theorem 2 and its proof.
\end{remark}

\begin{remark}
Theorem \ref{theoremXX} may be generalized to the setting of the $k$-plane transform for all $1\leq k\leq d-1$, although now the expression for the constant $C$ is in general less explicit, involving an integral of a negative power of the Brascamp--Lieb constant 
\eqref{mysteryconstant}
with respect to the $k$-dimensional subspaces $\pi_1,\hdots,\pi_{n}$. That this constant $C$ is positive at least follows from \eqref{finn} and the stability theory of Brascamp--Lieb constants \cite{BBCF2017}. We leave the details to the interested reader. 
\end{remark}

As may be expected, the arguments of this section may be combined with the adjoints of more general inequalities of Brascamp--Lieb type, leading to lower bounds on other tomographic transforms. In Section \ref{Section:nonlinear} we illustrate this with an example related to the spherical Radon transform.


\section{Nonlinear analogues}\label{Section:nonlinear}
The arguments that we use to establish the adjoint Brascamp--Lieb inequalities in Section \ref{adj-sec} make no specific use of the linear or algebraic structure of the underlying measure spaces. This is already aparent in Section \ref{Section:loomis}, where the adjoint Finner--Brascamp--Lieb inequality is presented in the setting of general (product) measure spaces. As a result various generalisations of the Brascamp--Lieb inequality \eqref{blc} known for functions defined on manifolds, or for nonlinear maps $B_j$, may be seen to admit such adjoints. In this section we formulate our adjoint Brascamp--Lieb inequality in a suitably abstract setting and illustrate its applicability with some examples.

Let $(\Omega,\mathcal{S},\mu)$ and $(\Omega_i,\mathcal{S}_i,\mu_i)$ be measure spaces and suppose that $\phi_i:\Omega\rightarrow\Omega_i$ is a measurable function for each $1\leq i\leq k$. 
Suppose further that the pushforward of the measure $\mu$ by $\phi_i$ is absolutely continuous with respect to $\mu_i$ for each $i$.
For a nonnegative measurable function $f:\Omega\rightarrow\mathbb{R}$ we may then define its pushforward $(\phi_i)_*f:\Omega_i\rightarrow\mathbb{R}$ via the formula
\begin{equation}\label{abstractpush}
\int_{\Omega_i}(\phi_i)_*f(x_i)g(x_i)d\mu_i(x_i)=\int_{\Omega}f(x)g(\phi_i(x))d\mu(x).
\end{equation}

\begin{theorem}[Abstract adjoint inequality]\label{theorem:abstract}
Suppose $0<p\leq 1$ and the exponents $p_i$ satisfy \eqref{pi-def}. If there exists a finite constant $\mathcal{C}$ such that
\begin{equation}\label{abstractBL}
\int_\Omega\prod_{i=1}^kf_i(\phi_i(x))^{c_i}d\mu(x)\leq \mathcal{C}\prod_{i=1}^k\left(\int_{\Omega_j}f_i(x_i)d\mu_i(x_i)\right)^{c_i}
\end{equation}
for all nonnegative functions $f_i$ on $\Omega_i$, then
\begin{equation}\label{abstractABL}
\|f\|_{L^p(d\mu)}\leq \mathcal{C}^{\frac{1}{p}-1}\prod_{i=1}^k\|(\phi_i)_*f\|_{L^{p_i}(d\mu_i)}^{\theta_i}
\end{equation}
for all nonnegative functions $f$ on $\Omega$.
\end{theorem}
The proof of this theorem is a routine adaptation of the arguments in Sections \ref{Section:loomis} and \ref{adj-sec}, and is left to the reader.

Applying Theorem \ref{theorem:abstract} to the \emph{nonlinear Brascamp--Lieb inequality} of \cite{BBBCF2020} we obtain the following.
\begin{corollary} If $B_i:\mathbb{R}^d\rightarrow\mathbb{R}^{d_i}$ are smooth submersions in a neighbourhood of a point $x_0\in\mathbb{R}^d$ and $\varepsilon>0$, then there exists $\delta>0$ such that
\begin{equation}\label{nonlinearABL}
\|f\|_{L^p(B(x_0,\delta))}\leq (1+\varepsilon)\BL(\mathrm{d}\mathbf{B}(x_0),\mathbf{c})^{\frac{1}{p}-1}\prod_{i=1}^k\|(B_i)_*f\|_{L^{p_i}(\mathbb{R}^{d_i})}^{\theta_i}
\end{equation}
for all nonnegative measurable functions $f$ on $\mathbb{R}^d$.
\end{corollary}
The inequality \eqref{nonlinearABL} is of course local in nature and may be viewed as a perturbation of the adjoint Brascamp--Lieb inequality \eqref{fd} where the underlying euclidean spaces are effectively replaced with smooth manifolds, and the linear surjections $B_i$ replaced with smooth submersions in a neighbourhood of a point. In certain situations, where the manifolds (and mappings $B_i$) are suitably symmetric, global statements with sharp constants are also available -- see for example \cite{Bramati2019}, \cite{BCEM2006}, \cite{BCELM2010}. A notable example arises from the analogue of Young's inequality on the sphere due to Carlen, Lieb and Loss \cite{CLL2004}. For $d\geq 2$ this states that 
\begin{equation}\label{cllinequality}
\int_{\mathbb{S}^{d-1}}\prod_{i=1}^d f_i(x_i)^{\frac{1}{2}}d\sigma(x)\leq\prod_{i=1}^d\left(\int_{-1}^1 f_i(t)d\nu(t)\right)^{\frac{1}{2}},
\end{equation}
where $d\sigma$ is the normalised surface measure on $\mathbb{S}^{d-1}$ and $d\nu$ is the (common) pushforward of $d\sigma$ by the spherical coordinate maps $\pi_i(x):=x_i$; explicitly, 
$d\nu(t)=\frac{|\mathbb{S}^{d-2}|}{|\mathbb{S}^{d-1}|}(1-t^2)^{\frac{d-3}{2}}dt$. This inequality is sharp, and equality holds if and only if each $f_i$ is a constant function. In addition to being global, \eqref{cllinequality} differs from \eqref{nonlinearABL} in that the curvature of the underlying manifold plays a role. 
Applying Theorem \ref{theorem:abstract} to \eqref{cllinequality} we obtain the following.

\begin{corollary}[Adjoint Carlen--Lieb--Loss]\label{adjcll}
Suppose that $0<p\leq 1$ and $\theta_1,\hdots,\theta_d$ are positive and sum to $1$. If $p_1,\hdots, p_d$ are such that $\theta_i\left(1-\frac{1}{p_i}\right)=\frac{1}{2}\left(1-\frac{1}{p}\right)$ for all $1\leq i\leq d$, then
\begin{equation}\label{acll}
\|f\|_{L^p(d\sigma)}\leq\prod_{i=1}^d\|(\pi_i)_*f\|_{L^{p_i}(d\nu)}^{\theta_i}
\end{equation}
for all nonnegative measurable functions $f$ on $\mathbb{S}^{d-1}$.
\end{corollary}
Corollary \ref{adjcll} may be interpreted as a statement about the spherical Radon-like averaging operator
$$
\mathcal{R}f(\omega,t):=\int_{\mathbb{S}^{d-1}}f(x)d\sigma_{\omega,t}(x),
$$
which averages a function $f:\mathbb{S}^{d-1}\rightarrow\mathbb{R}^+$ on the $(d-2)$-dimensional sphere $$\mathbb{S}^{d-1}\cap\{x\in\mathbb{R}^d:\omega\cdot x=t\}$$
for each for  $\omega\in\mathbb{S}^{d-1}$ and $t\in (-1,1)$; explicitly $$d\sigma_{\omega,t}(x)=\frac{|\mathbb{S}^{d-1}|}{|\mathbb{S}^{d-2}|}(1-t^2)^{-\frac{d-3}{2}}\delta(x\cdot\omega-t)d\sigma(x).$$  
Recalling that the pushforwards $(\pi_i)_*$ are defined by \eqref{abstractpush} with measures $d\mu:=d\sigma$ and $d\mu_i:=d\nu$, it follows that
\begin{equation}\label{piav}
(\pi_i)_*f(t)=\mathcal{R}f(e_i,t)
\end{equation}
for each $1\leq i\leq d$ and $t\in (-1,1)$. In particular, and as $d\sigma$ and $d\nu$ are probability measures, we see that equality holds in \eqref{acll} when $f$ is a constant function. 

\begin{remark}  A similar procedure can be applied to the discrete Brascamp--Lieb inequality on the symmetric group $S_n$ established in \cite[(1.4)]{carlen-hadamard}; we leave the details to the interested reader.
\end{remark}

Returning to the perspective of Section \ref{Section:X-ray}, we have the following simple application of Corollary \ref{adjcll}.
\begin{theorem}\label{theoremR}
Suppose that $d\geq 2$ and that $0<p,q\leq 1$ satisfy 
\begin{equation}\label{condR}
\frac{1}{d}\left(1-\frac{1}{q}\right)\geq \frac{1}{2}\left(1-\frac{1}{p}\right).
\end{equation}
Then
\begin{equation}\label{lowerR}
\|\mathcal{R}f\|_{L^q(d(\nu\times\sigma))}\geq \|f\|_{L^p(d\sigma)}
\end{equation}
for all nonnegative measurable functions $f$ on $\mathbb{S}^{d-1}$.
\end{theorem}
\begin{proof}
By H\"older's inequality it suffices to establish \eqref{lowerR} on the line 
\begin{equation}\label{line}
\frac{1}{d}\left(1-\frac{1}{q}\right)= \frac{1}{2}\left(1-\frac{1}{p}\right).
\end{equation}
For $x, \omega\in\mathbb{S}^{d-1}$ let $P_\omega x=x\cdot\omega$, the component of $x$ in the direction $\omega$. We caution that this notation differs from that in Section \ref{Section:X-ray}. Viewing $P_\omega$ as a mapping from $\mathbb{S}^{d-1}$ to $[-1,1]$, it follows from \eqref{piav} and a rotation that
\begin{equation}\label{link}
\mathcal{R}f(\omega,t)=(P_\omega)_*f(t)
\end{equation}
for each $t\in (-1,1)$.
Now, for each orthonormal basis $\{\omega_1,\hdots,\omega_d\}$ of $\mathbb{R}^d$ we have that
\begin{equation}\label{acllinv}
\|f\|_{L^p(d\sigma)}\leq\prod_{i=1}^d\|(P_{\omega_i})_*f\|_{L^{q}(d\nu)}^{\frac{1}{d}}
\end{equation}
by applying a suitable rotation to \eqref{acll}. Here we have also used \eqref{line}. By the AM-GM inequality it follows that
$$
\|f\|_{L^p(d\sigma)}^q\leq\frac{1}{d}\sum_{i=1}^d\|(P_{\omega_i})_*f\|_{L^{q}(d\nu)}^q.
$$
The inequality \eqref{lowerR} now follows on applying \eqref{link} and averaging over all choices of orthonormal basis $\{\omega_1,\hdots,\omega_d\}$.
\end{proof}
There are additional nonlinear Brascamp--Lieb inequalities not covered by the results of \cite{BBBCF2020}, in which the underlying linear Brascamp--Lieb constant is infinite, but the nonlinear bound still holds due to higher order nondegeneracy conditions.  In the case when $k=2$ and $d_1,d_2 = d-1$, for instance, estimates of this form (with the optimal range of exponents $c_1,c_2$) were established by Tao--Wright \cite{TW2002} and Stovall \cite{Stovall2011} (being closely related to $L^p$ improving estimates for Radon transforms along curves), and could be inserted into Theorem \ref{theorem:abstract} to obtain corresponding adjoint inequalities.  Similarly for the algebraic Brascamp--Lieb inequalities of Duncan \cite{duncan}.  We leave the details of these adjoint inequalities to the interested reader.


\section{Open questions}\label{Section:questions}
\begin{question}[Improved H\"older inequality for Gowers norms]  Does \eqref{fu2} hold for some $\eps>0$?
\end{question}

\begin{question}[Gaussian-stable data] What are the necessary and sufficient conditions for a Brascamp--Lieb datum $(\mathbf{B},\mathbf{c})$ to be gaussian-stable in the sense of Definition \ref{stability-def}?
\end{question}

\begin{question}[Lower bounds on restricted $k$-plane transforms]\label{mysteryBLconstant}
What is the natural generalization of Theorem \ref{theoremXE} to the $k$-plane transform?
\end{question}

\begin{question}[Sharp $k$-plane transform estimates]
Might we expect an explicit sharp form of Theorem \ref{theoremTk} in the spirit of the conjectures of Baernstein and Loss \cite{Baer1997}?
\end{question}

\begin{question}[A nonlinear adjoint Brascamp--Lieb inequality]
Can \eqref{nonlinearABL} be strengthened to
$$
\|f\|_{L^p(B(x_0,\delta))}\leq (1+\varepsilon)\ABL(\mathrm{d}\mathbf{B}(x_0),\mathbf{c},\theta,p)\prod_{i=1}^k\|(B_i)_*f\|_{L^{p_i}(\mathbb{R}^{d_i})}^{\theta_i}?
$$
\end{question}
\begin{question}[Density of gaussian extremisable data] Is the set of gaussian extremizable Brascamp--Lieb data $(\mathbf{B},\mathbf{c})$ dense in the set of all feasible Brascamp--Lieb data? This question is motivated by Remark \ref{Remark:bcct}.
\end{question}

\begin{question}[$L^p$ version of ABBN inequality] In Section \ref{Section:entropy} it was shown that the entropy Brascamp--Lieb inequalities from \cite{CCE2009} could be viewed as the limiting case (as $p \to 1$) of the adjoint Brascamp--Lieb inequalities established in this paper.  Could the monotonicity of Shannon entropy conjectured by Lieb \cite{lieb-entropy} and established in \cite{ABBN} similarly be the limiting case of some $L^p$ inequality?
\end{question}

\begin{question}[Relation to reverse Brascamp--Lieb inequalities]  Is there any relation or analogue of the adjoint Brascamp--Lieb inequalities to the reverse Brascamp--Lieb inequality of Barthe \cite{barthe}, or more generally to the forward--reverse Brascamp--Lieb inequality of Liu--Courtade--Cuff--Verd\'{u} \cite{LCCV}?
\end{question}


\end{document}